\documentclass[11pt,final]{amsart}
\usepackage{amsmath,amssymb,amsthm,amsfonts,mathrsfs,amsopn}
\usepackage[all]{xy}
\usepackage{dsfont}
\usepackage{hyperref}
\usepackage{color}
\usepackage{esint}
\usepackage{mathtools}
\mathtoolsset{showonlyrefs}
\usepackage{slashed}
\usepackage{showkeys}
\usepackage{tikz-cd}
\usepackage[margin=0.9in]{geometry}

\usepackage[obeyDraft]{todonotes}

\newcommand{\Abracket}[1]{\left<#1\right>} 

\newcommand{\R}{\mathbb{R}}
\newcommand{\C}{\mathbb{C}}
\newcommand{\bZ}{\mathbb{Z}}

\newcommand{\dd}{\mathop{}\!\mathrm{d}}

\newcommand{\p}{\partial}

\newcommand{\Id}{\mathrm{Id}}
\newcommand{\D}{\slashed{D}}

\newcommand{\snabla}{\slashed{\nabla}} 

\DeclareMathOperator{\Dom}{Dom}
\DeclareMathOperator{\dv}{\dd{vol}}

\DeclareMathOperator{\End}{End}

\DeclareMathOperator{\id}{Id}
\DeclareMathOperator{\Ker}{Ker}

\DeclareMathOperator{\Ran}{Ran} 
\DeclareMathOperator{\rank}{rank}

\DeclareMathOperator{\Span}{Span}

\newcommand{\arsinh}{\operatorname{arsinh}}
\newcommand{\Pthd}[3]{({#2*0.92+#1*0.38},{#3*0.70-#1*0.24})}
\newcommand{\Ptwo}[2]{({#1*1.10},{#2*0.70})}
\newtheorem{thm}{Theorem}[section]
\newtheorem{dfn}[thm]{Definition}
\newtheorem{lemma}[thm]{Lemma}

\newtheorem{cor}[thm]{Corollary}

\newtheorem{rmk}[thm]{Remark}

\usepackage[
    backend=biber,
    style=numeric,
    sorting=nyt
]{biblatex}

\title{Continuity of Weighted Dirac Spectra}

\author[Z. Qiu]{Zixuan Qiu}
\address{Zixuan Qiu, School of mathematics and statistics, Beijing Institute of Technology, Zhongguancun South Street No. 5, 100081 Beijing, P.R.China.}
\email{ZixuanQiu@bit.edu.cn}

\author[R. Wu]{Ruijun Wu}
\address{Ruijun Wu, School of mathematics and statistics, Beijing Institute of Technology, Zhongguancun South Street No. 5, 100081 Beijing, P.R.China.}
\email{ruijun.wu@bit.edu.cn}

\thanks{ R.W. is partially supported by the Beijing Natural Science Foundation No. 1262018.
}

\begin{document}
\begin{abstract}
For the weighted Dirac eigenvalue problem, we show that the two-sided weighted spectrum depends continuously on the weight under continuous deformations within a uniformly elliptic class. Moreover, for differentiable families of weights we obtain a quantitative Lipschitz estimate for the full spectrum in the $\arsinh$--metric, based on a weighted Hellmann--Feynman variational identity.
\end{abstract}

\maketitle


\section{Introduction}\label{sec:intro}

We consider a weighted eigenvalue problem for Dirac operators of the form
\begin{align}\label{eq:WD}
    \D_g\psi=\lambda\,A\psi \qquad \text{on } M^n,
\end{align}
where $(M^n,g)$ is an $n$-dimensional closed Riemannian spin manifold with a fixed spin structure $\Theta$, and~$\D_g$ is the Dirac operator on the associated spinor bundle $\Sigma_g M$.
A section~$\psi\in\Gamma(\Sigma_g M)$ is referred to as a spinor field.
The weight $A\in\End(\Sigma_g M)$ is assumed to be symmetric and positive definite fiberwisely.
If~\eqref{eq:WD} is satisfied by~$(\lambda,\psi)\in\R\times\Gamma(\Sigma_g M)$ with~$\psi\neq 0$, we say that~$\lambda$ is a weighted Dirac eigenvalue for the weight~$A$, and call~$\psi$ a weighted eigenspinor.
Similar to the spectral theory for~$\D_g$, it is readily seen that \eqref{eq:WD} has a two-sided unbounded discrete real spectrum, each eigenvalue has finite multiplicity, and that the eigenspinors form a complete orthonormal basis of the $L^2$-space of spinors.
Moreover, these weighted eigenvalues admit min-max characterizations and depend continuously on the weights, see~\cite{Qiu2025weighted}.

%
Weighted Dirac equations of the form \eqref{eq:WD} arise naturally in geometric and variational problems.
Within the local spinorial Weierstrass representation of surfaces in $\mathbb{R}^3$ \cite{Friedrich1998Representation,Taimanov1997modified,Taimanov1998Weierstrass}, the equation $\D_g\psi=H\psi$ arises, where $H$ denotes the mean curvature.
The main motivation for the present work comes from super Liouville-type systems \cite{Jevnikar2020existence,Jevnikar2021SuperLiouville,Jevnikar2021sinhGordon,Jost2007super},
where the spinor component solves $\D_g\psi=\lambda e^u\psi$ with~$\lambda>0$ and~$u\in H^1(M)$.
One needs to study the behavior of~$\lambda$ when~$u$ varies in~$H^1(M)$.
We also mention that, in the nonlinear spinorial Yamabe problem \cite{Bartsch2022Yamabe,Isobe2024Yamabe,Isobe2023Yamabe},
linearization leads to spectral equations of the form $\D_g\psi=\lambda|\varphi|^{\frac{2}{n-1}}\psi$, where
$\varphi\neq 0$ is a solution of the Yamabe-type spinorial equation. We restrict here to the linear eigenvalue problem \eqref{eq:WD}, allowing in particular scalar weights $A= f\Id$ with $f\in C^1(M,\mathbb{R}_+)$.
In this paper, we focus on uniformly positive definite weights, which excludes sign-changing or degenerate weights that may occur in some geometric applications.

When $A=\Id$, \eqref{eq:WD} reduces to the classical Dirac spectral problem, which has been a driving force in spin geometry; see e.g.\ \cite{Ginoux2009Dirac} and references therein. Under perturbations of the metric, a number of fine spectral properties of $\D_g$ have been studied, for instance by B\"ar \cite{Baer1996Metrics} and Dahl \cite{Dahl2003Dirac,Dahl2005Prescribing}. A major analytical difficulty in \emph{full-spectrum} continuity problems is that the Dirac spectrum is real and discrete but unbounded in both directions of $\mathbb{R}$; to compare \emph{all} eigenvalues simultaneously one needs a global encoding together with a suitable metric on the space of spectra. This was achieved by Nowaczyk~\cite{Nowaczyk2013continuity}, who represents the spectrum by a monotone map $\mathbb{Z}\to\mathbb{R}$ and equips the space of such maps with an $\arsinh$-metric, obtaining global continuity for the Dirac spectrum under metric perturbations; see also \cite{Lott2002Collapsing}. In a different direction, Roos studied Dirac operators with symmetric $W^{1,\infty}$ potentials along collapsing sequences of spin manifolds and identified the corresponding limit operators~\cite{Roos2018W1Inf,Roos2020Collapse}.

These developments inspire our use of a full-spectrum metric and highlight the robustness issues that arise under low-regularity perturbations. In the present work, we therefore consider a closely related but complementary perturbation problem: we keep the underlying Dirac operator $\D_g$ fixed and study continuity properties of the \emph{weighted} spectrum as the fiberwise positive definite endomorphism~$A$ varies. Our emphasis is on \emph{full-spectrum} continuity, i.e. simultaneous control of the entire two-sided spectrum counted with multiplicity, and on low regularity of the weights.


We fix $p>n$ and denote
\begin{align}
\mathcal A:=W^{1,p}\bigl(M,\End(\Sigma_gM)\bigr),
& &
\mbox{ and } \quad
\|F\|_{\mathcal A}:=\|F\|_{L^p}+\|\nabla^gF\|_{L^p},\quad \forall F\in\mathcal A.
\end{align}
Moreover, for $0<\Lambda_1\le \Lambda_2$, let
\begin{align}\label{eq:def_P_C1C0_intro}
\mathcal P_{\Lambda_1,\Lambda_2}
:=\Bigl\{A\in\mathcal A:\ A(x)=A(x)^*,\ \Lambda_1\Id\le A(x)\le \Lambda_2\Id\Bigr\}.
\end{align}
The weights in $\mathcal P_{\Lambda_1,\Lambda_2}$ are said to be admissible. The space of smooth sections of $\Sigma_gM\rightarrow M$ is denoted by $\Gamma(\Sigma_gM)$ while those sections of regularity $W^{k,p}$ are denoted by $W^{k,p}(M, \Sigma_gM)$. For example,~$L^2(M, \Sigma_gM)$ is a Hilbert space equipped with the standard $L^2$ inner product and $H^1(M, \Sigma_gM)$ consists of the sections of Sobolev regularity $H^1 = W^{1,2}$.

Each $ A \in \mathcal{P}_{\Lambda_1,\Lambda_2} $ induces a weighted inner product given by
\begin{align}\label{eq:wei_inn_pro}
(\psi, \phi)_A := \int_M \langle A\psi, \phi \rangle \, dv_g, \quad \forall \psi, \phi \in L^2(M, \Sigma_gM),
\end{align}
and we write $\mathcal{H}_A := \bigl(L^2(M, \Sigma_gM), (\cdot, \cdot)_A\bigr)$ and view
\begin{align}
B(A) := A^{-1}\D
\end{align}
as an operator on $\mathcal{H}_A$ with domain $H^1(M, \Sigma_gM)$. Additionally, we consider the conjugated operator
\begin{align}
\widetilde{D}(A) := A^{-1/2}\D A^{-1/2},
\quad
\text{with }\Dom\bigl(\widetilde{D}(A)\bigr) = H^1(M, \Sigma_gM).
\end{align}

We use the space $\mathfrak{Mon}$ and the uniform $\arsinh$--metric as in \cite{Nowaczyk2013continuity}.
\begin{dfn}\label{dfn:Mon}
Let $\mathfrak{Mon}$ be the set of all functions $u:\bZ\to\R$ such that:
\begin{enumerate}
\item[(1)] $u$ is nondecreasing, so $u(j)\le u(j+1)$ for all $j\in\bZ$,

\item[(2)] $u$ is proper, so $u(j)\to -\infty$ as $j\to -\infty$ and $u(j)\to +\infty$ as $j\to +\infty$.
\end{enumerate}
The $\arsinh$-metric on $\mathfrak{Mon}$ is defined by
\begin{equation}\label{eq:da}
 d_a(u,v):=\sup_{j\in\bZ}\Big|\arsinh\bigl(u(j)\bigr)-\arsinh\bigl(v(j)\bigr)\Big|,
 \qquad \forall u,v\in \mathfrak{Mon}.
\end{equation}
\end{dfn}
By Weyl asymptotics for weight eigenvalues \cite{Qiu2025weighted}, the ordered eigenvalues of Dirac-type operators belong to $\mathfrak{Mon}$. For an admissible weight $A$ we enumerate the spectrum of $\widetilde D(A)$ by a nondecreasing map.
\begin{dfn}\label{def:ordered_spectrum_counting}
For $A \in\mathcal P_{\Lambda_1,\Lambda_2}$, let $\mathfrak s^{A}:\bZ\to\R$ be the unique nondecreasing function such that
\begin{enumerate}
    \item [(1)] $\mathfrak{s}^{A}(\bZ) = \operatorname{Spect}\bigl(A^{-1/2}\D\,A^{-1/2}\bigr)$;
    \item [(2)] for every $\lambda\in\R$,
\begin{align}
\dim\Ker\bigl(A^{-1/2}\D\,A^{-1/2}-\lambda\bigr)
=\sharp\bigl(\mathfrak s^{A}\bigr)^{-1}(\lambda),
\end{align}
where $\sharp$ denotes the cardinality of a set, so $\sharp\bigl(\mathfrak s^{A}\bigr)^{-1}(\lambda)$ is the multiplicity of $\lambda$;
    \item [(3)] $\mathfrak s^{A}(0)\geq 0$ and $\mathfrak s^{A}(-1)< 0$.
\end{enumerate}
\end{dfn}
At first glance, this may be a good spectral map, which was expected to behave nicely with respect to continuous deformations of the Dirac operator.
However, this is in general not the case even for a smooth deformation of the Riemannian metrics, as already observed in~\cite{Nowaczyk2013continuity}.
The continuity properties of the spectral maps for spin Dirac operator with respect to~$C^1$ deformations of Riemannian metrics were subtle and was achieved by passing to a quotient space of~$\mathfrak{Mon}$ in~\cite{Nowaczyk2013continuity}, and the resulting spectral map is rather implicit.
Fortunately here we can show that the potential obstruction for continuity does not appear in our problem and the map~$\mathfrak{s}^A$ turns out to be continuous, sometimes even differentiable, under suitable but natural assumptions on the weights.

To state the main results we use the following hypotheses:
\begin{equation}\label{hyp:H0}\tag{H0}
\text{The family } A_t \text{ lies in } \mathcal P_{\Lambda_1,\Lambda_2} \text{ and is continuous in } t.
\end{equation}
\begin{align}\label{hyp:H1}\tag{H1}
& I\subset\R\ \text{is an interval and }I\ni t\mapsto A_t\in \mathcal P_{\Lambda_1,\Lambda_2}\ \text{is of class }C^1.
\end{align}
Our first result addresses the continuity of the full weighted spectrum under \emph{uniform} perturbations of~$A$.
Although we work on $W^{1, p}(M, \Sigma_gM)$ in the analysis, the continuity statement can be stated with respect to the $C^0$-topology on the uniformly elliptic class \eqref{eq:def_P_C1C0_intro}.

\begin{thm}\label{thm: C0_CONT}
Assuming \eqref{hyp:H0}, the spectral map
\begin{align}
    \mathfrak s:\bigl(\mathcal P_{\Lambda_1,\Lambda_2},\|\cdot\|_{\mathcal A}\bigr)\longrightarrow (\mathfrak{Mon},d_a),
    \qquad
    A\longmapsto \mathfrak{s}^A,
\end{align}
is continuous.
\end{thm}

The idea of the proof is as follows. We first transfer the weighted problem to a self-adjoint Dirac-type operator on a fixed Hilbert space by a canonical isometric conjugation depending on $A$. Working in $W^{1,p}(M, \End(\Sigma_gM))$, we obtain quantitative control of
$A^{\pm 1/2}$ and their covariant derivatives via a Sylvester-type equation, which yields operator-norm continuity of the
conjugated operators. This places the family into Nowaczyk's framework~\cite{Nowaczyk2013continuity} of discrete self-adjoint families of type (A), so that full-spectrum continuity in the $\arsinh$-metric follows locally. Finally, the indexing is shown to be stable by using constancy of the kernel dimension.

\

Our second result is quantitative: for a $C^1$ family of weights~$(A_t)$ in $\mathcal P_{\Lambda_1,\Lambda_2}$, the full spectrum is locally Lipschitz in~$t$ with respect to the $\arsinh$-metric.
The main tool is the Hellmann–Feynman variational identity, which is recalled in Section~\ref{sec:dirac-spectrum-perturbation}. Here we need to adapt it to the weighted setting.

\begin{thm}\label{thm:C1_CONT}
Assuming \eqref{hyp:H1}, namely that the map $I\ni t\mapsto \|\dot{A_t}\|_{W^{1,p}(\End(\Sigma_gM))}$ is continuous, where $\dot{A_t} := \frac{\partial}{\partial t}A_t$.
Denote
\begin{align}\label{eq:LJ_def_C1}
L_I:=\frac{C}{\Lambda_1}\sup_{t\in I}\|\dot{A_t}\|_{\mathcal A},
\end{align}
where $C$ is the Sobolev constant for the embedding of $W^{1,p}\bigl(M,\End(\Sigma_gM)\bigr)$ into $L^{\infty}\bigl(M,\End(\Sigma_gM)\bigr)$.
Let $\mathfrak s_t\in\mathfrak{Mon}$ denote the spectral tuple of the weighted problem $\D\psi=\lambda A_t\psi$.
\begin{enumerate}
\item [(1)]  If $t_0\in I$ and $(\lambda(t_0),\varphi(t_0))$ is an eigenpair of $B_{t_0}:=A_{t_0}^{-1}\D$,
then there exist a neighborhood $U\subset I$ of $t_0$ and $C^1$ maps
$t\mapsto(\lambda(t),\varphi(t))$ with $\varphi(t)\in H^1(M,\Sigma_gM)$ such that
\begin{align}
\D\varphi(t)=\lambda(t)A_t\varphi(t),\qquad
\int_M\langle A_t\varphi(t),\varphi(t)\rangle\,dv_g=1,\qquad t\in U,
\end{align}
and
\begin{align}
\lambda'(t)=-\lambda(t)\int_M\langle \dot{A_t}\varphi(t),\varphi(t)\rangle\,dv_g,
\qquad t\in U.
\end{align}
In particular, $|\lambda'(t)|\le L_I|\lambda(t)|$ for all $t\in U$.
\smallskip
\item [(2)] For any $s,t\in I$,
\begin{align}\label{eq:MonLip_intro}
d_a(\mathfrak s_t,\mathfrak s_s)\le L_I\,|t-s|.
\end{align}
Consequently, $t\mapsto\mathfrak s_t$ is Lipschitz as a map from $I$ to $(\mathfrak{Mon},d_a)$.
\end{enumerate}

\end{thm}

Note that we assumed the~$C^1$ regularity of the eigenpair~$(\lambda(t),\varphi(t))$ in the statement above, which does not follow automatically, because of the eigenvalues of higher multiplicity for Dirac operators and also the possible crossing phenomena for eigenvalues, see discussion in Section~\ref{sec:C1weight}. For the same reason, the Lipschitz regularity of the spectral map~$\mathfrak{s}_t$ in~$t$ is the best that one can expect.

We analyze the weighted eigenvalue problem and differentiate along a $C^1$ eigenpair under an $A_t$-dependent normalization. This process produces a weighted Hellmann--Feynman variational identity, which subsequently establishes a uniform Lipschitz bound for $\arsinh(\lambda(t))$ across any eigenvalue. To extend from a single eigenvalue to the ordered full spectrum, we localize within finite spectral windows and employ finite-rank spectral projections to reduce the problem to the realm of finite-dimensional perturbation theory, ensuring Lipschitz control over eigenvalues even at crossings. A sorting stability argument transfers the multiset bounds to monotone enumeration, and as the spectral window expands to $\pm\infty$, we obtain the global estimate.

\
Several remarks are warranted. Although the continuity statement is formulated for $C^0$ perturbations within the uniformly elliptic class~$\mathcal P_{\Lambda_1,\Lambda_2}$, the argument is carried out in $W^{1,p}$. This is the natural regime for our present method: we first conjugate the weighted problem to a self-adjoint Dirac-type operator $\widetilde D(A)$ on a fixed Hilbert space, and then apply Nowaczyk's spectrum continuity principle for discrete self-adjoint families of type (A). In this setup one needs access to first-order weak derivatives of $A$, which is provided precisely by the $W^{1,p}$ assumption with $p>n$. This relaxes the $W^{1,\infty}$-type hypotheses in related works such as Roos~\cite{Roos2018W1Inf}, but it remains restrictive: weights of merely $L^p$ regularity, as may occur in super Liouville-type systems, are not covered by the present approach. We expect that further extensions to coarser weights will require a more direct perturbation argument in the spirit of Kato's theory~\cite{Kato1995perturbation}, and we hope to address this in a future work.

\

The paper is organized as follows.
In Section~\ref{sec:dirac-spectrum-perturbation}, we develop the perturbation-theoretic framework used throughout the paper: spectral
encoding of two-sided unbounded discrete spectra, Nowaczyk's full-spectrum continuity principle for type-(A)
families, and the weighted Hellmann--Feynman identity. In Section~\ref{sec:C0weight} we prove
Theorem~\ref{thm: C0_CONT}, establishing full-spectrum continuity under $C^0$ perturbations of uniformly elliptic
weights. Section~\ref{sec:C1weight} is devoted to the proof of Theorem~\ref{thm:C1_CONT}, providing the $\arsinh$-Lipschitz
control for $C^1$ parameter families. Finally, in Section~\ref{sec:discussion}, we discuss how the results extend beyond the Dirac operator to general first-order formally self-adjoint elliptic operators.

\

\section{Preliminaries}\label{sec:dirac-spectrum-perturbation}
In this section, we recall two analytic tools that will be used throughout the paper: the full-spectrum continuity framework and the Hellmann--Feynman variational identity.

\subsection{Encoding of Dirac spectra}\label{sec:dirac-spectrum-continuity}

In this subsection, we metrize the Dirac spectral configurations and recall the local continuity theorem, which will form the foundation for the subsequent analysis.

Let $(M^n,g)$ be an $n$-dimensional closed Riemannian spin manifold with a fixed spin structure $\Theta$, and let $\Sigma_g M$ be the associated spinor bundle.
Following the notation introduced earlier, we denote by $\D$ the Dirac operator corresponding to the metric $g$.
We regard $\Sigma_g M$ as a real vector bundle of rank $2^{\left[\frac{n+1}{2}\right]}$ and, for the moment,
suppress its Hermitian structure. The bundle carries the natural Riemannian data: a fiberwise inner product $g^s$,
a spin connection $\snabla$, and a Clifford multiplication $\gamma$ satisfying the Clifford relations
\begin{align}
    \gamma(X)\gamma(Y)+\gamma(Y)\gamma(X)= -2g(X,Y)\id_{\Sigma_g M},
    \qquad \forall X,Y\in\Gamma(TM).
\end{align}
These structures are compatible, so that $(\Sigma_g M, g^s, \snabla, \gamma)$ is a Dirac bundle in the sense of
\cite[Definition~5.1]{LawsonMichelsohn1989spin}. Accordingly, the Dirac operator is defined as the composition
\begin{align}
    \Gamma(\Sigma_g M) \xrightarrow[]{\snabla}\Gamma(T^*M\otimes \Sigma_g M)
    \xrightarrow{\cong} \Gamma(TM\otimes \Sigma_g M)
    \xrightarrow{\gamma} \Gamma(\Sigma_g M).
\end{align}
In a local orthonormal frame $\{e_i\}_{i=1}^n$, this reads
\begin{align}
    \D\psi =\sum_{i} \gamma(e_i)\snabla_{e_i} \psi,
    \qquad \forall \psi\in \Gamma(\Sigma_g M).
\end{align}
Since $M$ is closed and $\D$ is a first-order elliptic operator, $\D$ is essentially self-adjoint on the Hilbert
space of $L^2$-spinors and has compact resolvent. In particular, $\operatorname{Spect}(\D)\subset\mathbb{R}$ is discrete, each eigenvalue has finite multiplicity, and the spectrum is unbounded in both directions in~$\R$;
see, for instance,~\cite{Friedrich2000Dirac,Ginoux2009Dirac}.

To track and compare two-sided unbounded discrete spectra, we follow~\cite{Nowaczyk2013continuity}
and work with monotone spectral enumerations in $\mathfrak{Mon}$ and the metric $d_a$ introduced in Definition~\ref{dfn:Mon}.
Note that the additive integer group $\bZ$ acts on $\mathfrak{Mon}$ by shifts: for any~$u\in\mathfrak{Mon}$ and~$k\in\bZ$, the~$k$-shift of~$u$ is given by
\begin{align}\label{eq:shift_action}
(u\cdot k)(j):=u(j+k),\qquad \forall j\in\bZ.
\end{align}
These shifts are isometries, so we can consider the quotient space
\begin{align}
\mathfrak{Conf}:=\mathfrak{Mon}/\bZ
\end{align}
and let~$\pi:\mathfrak{Mon}\to\mathfrak{Conf}$ be the quotient map.
For~$u\in\mathfrak{Mon}$ we write~$\bar u:=\pi(u)$.
Moreover, the space~$\mathfrak{Conf}$ is then equipped with the induced quotient metric
\begin{align}\label{eq:quotient_metric}
\bar d_a(\bar u,\bar v):=\inf_{k\in\bZ} d_a\bigl(u, v\cdot k\bigr),
\qquad
u\in\bar u,\ v\in\bar v.
\end{align}

\

We now recall the abstract setting of discrete self-adjoint families of type (A), using the terminology of~\cite{Nowaczyk2013continuity}, in order to state the relevant local full-spectrum continuity result.

Let $\mathcal H$ be a Hilbert space and let $\mathcal E$ be a topological space. Denote by $\mathcal C(\mathcal H)$ the set of all closed, densely
defined linear operators on $\mathcal H$.
Let $\mathcal X$ be a normed linear space, denote by $\mathcal B(\mathcal X,\mathcal H)$ the Banach space of all bounded linear operators from $\mathcal X$ to $\mathcal H$, here and in what follows, $\|\cdot\|_{\mathcal B}$ denotes the operator norm.

\begin{dfn}\label{def:typeA-family}
 A map $T:\mathcal E\to\mathcal C(\mathcal H)$, $e\mapsto T_e$, is called a \emph{self-adjoint family of type (A)} if:
\begin{enumerate}
\item there exists a dense subspace $\mathcal Z\subset\mathcal H$ such that $\Dom(T_e)=\mathcal Z$ for all $e\in\mathcal E$;
\item for each $e\in\mathcal E$, the operator $T_e$ is self-adjoint;
\item there exists a norm $|\cdot|$ on $\mathcal Z$ such that, for each $e\in\mathcal E$, the operator
      $T_e:(\mathcal Z,|\cdot|)\to (\mathcal H,\|\cdot\|_{\mathcal H})$ is bounded and the graph norm of $T_e$ is equivalent to $|\cdot|$;
\item the map $\mathcal E\to \mathcal B(\mathcal Z,\mathcal H)$, $e\mapsto T_e$, is continuous with respect to the operator norm on $\mathcal B(\mathcal Z,\mathcal H)$.
\end{enumerate}
If, in addition, each $T_e$ has compact resolvent, we call the family \emph{discrete}.
\end{dfn}
The following theorem applies to discrete self-adjoint families of type (A). For conciseness, we denote~$\mathfrak{s}_T^e := \mathfrak{s}_{T_e}$.
\begin{thm}[{\cite[Theorem 4.10]{Nowaczyk2013continuity}}]\label{thm: key expanding}
Let $T:\mathcal E \rightarrow \mathcal C (\mathcal H)$ be a discrete self-adjoint family of type (A). Then for any $e_0 \in\mathcal E$ and any
$\varepsilon>0$ there exists an open neighborhood $U \subset\mathcal E$ of $e_0$ such that
\begin{align}
\forall e \in U\ \exists k \in \mathbb{Z}\ \forall j \in \mathbb{Z}:\
d_a\!\left(\mathfrak{s}_T^{e_0}(j), \mathfrak{s}_T^e(j+k)\right)<\varepsilon,
\end{align}
here and in what follows, for $x,y\in\mathbb{R}$, $d_a(x,y):=|\arsinh(x)-\arsinh(y)|$.
\end{thm}

\begin{rmk}
\label{rem:dirac-spectrum-continuity}
Let $\mathcal{R}(M)$ denote the space of all Riemannian metrics on $M$, endowed with the $C^1$-topology.
By the standard isometric identifications of spinor bundles, after fixing a reference metric and transporting
$L^2$-spinors to a common Hilbert space, the assignment
\begin{align}
g\longmapsto \D_g
\end{align}
can be viewed as a family of operators on a fixed Hilbert space. Its common domain can be taken as the corresponding
$H^1$-space, and the family satisfies Definition~\ref{def:typeA-family}; moreover, each $\D_g$ has compact
resolvent, so the family is discrete. Hence Theorem~\ref{thm: key expanding} implies that for any
$g_0\in\mathcal{R}(M)$ and any $\varepsilon>0$ there exists a neighborhood $U$ of $g_0$ such that for every
$g\in U$ there exists $k\in\mathbb{Z}$ with
\begin{align}
\forall j\in\mathbb{Z}:\quad
d_a\!\left(\mathfrak{s}_{\D_{g_0}}(j),\,\mathfrak{s}_{\D_g}(j+k)\right)<\varepsilon .
\end{align}
In other words, the spectral configuration map
\begin{align}
(\mathcal{R}(M),C^1)\longrightarrow (\mathfrak{Conf},\bar d_a),\qquad g\longmapsto [\mathfrak{s}_{\D_g}]
\end{align}
is continuous at $g_0$.
\end{rmk}

\subsection{Hellmann--Feynman variational identity}\label{subsec:HF}

Under \eqref{hyp:H1},
we will differentiate eigenvalues of the weighted Dirac operators. The resulting first variation formula is an instance of the classical \emph{Hellmann--Feynman variational identity},  which arises from quantum mechanism~\cite{Feynman1939Force}.
For recent discussions and extensions beyond the self-adjoint setting, see e.g.\ \cite{Hajong2024Hellmann}.

Let $\mathcal H$ be a complex Hilbert space and $I\subset\mathbb R$ an open interval.
Let
\begin{align}
H(t): \mathcal{Z}\subset \mathcal H\to \mathcal H,\qquad t\in I,
\end{align}
be a family of self-adjoint operators with a common dense domain $\mathcal{Z}$.
Assume that $H(t)$ is $C^{1}$, i.e.\ for every $u\in \mathcal{Z}$
the map $t\mapsto H(t)u\in\mathcal H$ is of class $C^{1}$, and define~$\dot{H}(t)$ by
\begin{align}
\dot H(t)u:=\partial_t\bigl(H(t)u\bigr),\qquad u\in\mathcal{Z}.
\end{align}

Assume that $H(t)$ admits a simple eigenvalue branch $\lambda_n(t)\in\mathbb R$,
with associated rank-one spectral projection $P_n(t)$, and let $\psi_n(t)\in \mathcal{Z}$ be a
locally $C^{1}$ choice of normalized eigenvector spanning $\Ran P_n(t)$:
\begin{align}
H(t)\psi_n(t)=\lambda_n(t)\psi_n(t),
\qquad
\|\psi_n(t)\|_{\mathcal H}=1.
\end{align}
Then $\lambda_n$ is differentiable and satisfies the Hellmann--Feynman identity
\begin{equation}\label{eq:HF}
\lambda_n'(t)=\big\langle \psi_n(t),\,\dot H(t)\psi_n(t)\big\rangle_{\mathcal H}.
\end{equation}
Equivalently, since $P_n(t)$ is the orthogonal projection onto $\mathrm{span}\{\psi_n(t)\}$, we have
\begin{align}
P_n(t)u=\langle u,\psi_n(t)\rangle_{\mathcal H}\,\psi_n(t),\qquad u\in\mathcal H.
\end{align}
Hence, we have
\begin{equation}\label{eq:HF-trace}
\lambda_n'(t)=\operatorname{Tr}\!\big(P_n(t)\dot H(t)\big).
\end{equation}

If $H(t_0)$ has an eigenvalue $\lambda_0$ of finite multiplicity $m$, let~$P(t_0)$ be the associated spectral projection. For an orthonormal basis~$\{\phi_i\}_{i=1}^m$ of~$\Ran P(t_0)$ define the Hermitian matrix
\begin{align}
\mathsf M_{ij}:=\big\langle \phi_i,\,\dot H(t_0)\phi_j\big\rangle_{\mathcal H},\qquad 1\le i,j\le m.
\end{align}
Then the first-order slopes of the eigenvalue branches through $t_0$ are given by the eigenvalues of the finite-dimensional operator~$P(t_0)\dot H(t_0)P(t_0)$.

For any $C^{1}$ family of normalized vectors $\psi(t)\in\mathcal{Z}$ with
$\|\psi(t)\|_{\mathcal H}=1$, set
\begin{align}
\lambda(t):=\big\langle \psi(t),\,H(t)\psi(t)\big\rangle_{\mathcal H}.
\end{align}
A direct differentiation yields the exact identity
\begin{equation}\label{eq:trial-derivative}
 \lambda'(t)=\big\langle \psi(t),\,\dot H(t)\psi(t)\big\rangle_{\mathcal H}
+2\,\Re\big\langle \dot{\psi}(t),\,(H(t)- \lambda(t))\psi(t)\big\rangle_{\mathcal H}.
\end{equation}
Here we used the fact that~$\lambda(t)$ is real and hence~$\Re\langle\dot{\psi}(t),\lambda(t)\psi(t)\rangle_{\mathcal{H}}= \lambda(t)\Re\langle\dot{\psi}(t),\psi(t)\rangle_{\mathcal{H}}= 0$.
For an eigenvector, $(H(t)- \lambda(t))\psi(t)=0$ and \eqref{eq:trial-derivative}
reduces to \eqref{eq:HF}. In variational approximations based on $t$-dependent trial subspaces, the second term in
\eqref{eq:trial-derivative} does not vanish in general.
It quantifies the additional contribution caused by the $t$-dependence of the chosen trial family
$\psi(t)$ through the residual~$(H(t)-\lambda(t))\psi(t)$.

\section[C0 weight]{\texorpdfstring{The Continuity of Weighted Spectra with $C^0$ Weight}{The Continuity of Weighted Spectra with C0 Weight}}\label{sec:C0weight}

For each $A\in\mathcal P_{\Lambda_1,\Lambda_2}$, we work with the conjugated self-adjoint Dirac-type operator on the fixed Hilbert space $L^2(M, \Sigma_gM)$
\begin{align}
\widetilde D(A):=A^{-1/2}\D A^{-1/2},\qquad \Dom(\widetilde D(A))=H^1(M, \Sigma_gM),
\end{align}
and denote by $\mathfrak{s}^A\in\mathfrak{Mon}$ its associated ordered spectrum. The goal of this section is to prove Theorem~\ref{thm: C0_CONT}, namely, the continuity of the map~$A\mapsto \mathfrak{s}^A$ with respect to the $C^0$-topology on the space $\mathcal P_{\Lambda_1,\Lambda_2}$ and the metric $d_a$. We consider weights of class~$W^{1,p}$ for some $p>n$, although this may not be the optimal regularity assumption. Note that, according to the Sobolev embedding theorem, $W^{1,p}\hookrightarrow L^{\infty}$ for~$p>n$. This allows us to obtain uniform control of $A^{\pm1/2}$ and their derivatives, and to prove that the map $A\mapsto \widetilde D(A)$ is continuous in the operator norm as a map into~$\mathcal B(H^1(M, \Sigma_gM),L^2(M, \Sigma_gM))$.

We begin with a Sylvester-type equation that provides pointwise bounds for derivatives of $A^{\pm1/2}$.

\begin{lemma}[{\cite[Theorem 9.2]{Rajendra1997AX-XBequalsY}}]\label{lem: Sylvester solution}
Let $A$ and $B$ be operators whose spectra are contained in the open right half plane and the open left half plane, respectively.
Then the unique solution of the equation $AX-XB=Y$ is expressed as
\begin{align}
X=\int_0^{\infty} e^{-tA} Y e^{tB} \, dt.\notag
\end{align}
\end{lemma}
The special case relevant for our application arises when the unknown $X$ is coupled to two
fiberwise positive self-adjoint endomorphisms through a \emph{symmetrized} Sylvester map.
Let $S,T$ be fiberwise self-adjoint endomorphisms and assume that there exists $m>0$ such that
$S\ge m\,\Id$ and $T\ge m\,\Id$.
For each endomorphism~$X$, consider the linear map
\begin{align}
\mathcal L_{S,T}(X):=SX+XT.\notag
\end{align}
Since the spectrum of $S$ satisfies $\sigma(S)\subset [m,\infty)$ and $\sigma(-T)\subset (-\infty,-m]$, Lemma~\ref{lem: Sylvester solution}
applies to the equation $SX-X(-T)=Y$. In particular, $\mathcal L_{S,T}$ is invertible and the unique
solution of $\mathcal L_{S,T}(X)=Y$ is given by
\begin{equation}\label{eq:sylvester_inverse_formula_sec3}
X=\int_0^\infty e^{-sS}\,Y\,e^{-sT}\,ds.
\end{equation}
Moreover, the positivity of $S$ and $T$ implies $\|e^{-sS}\|_{\mathcal B}\le e^{-sm}$ and
$\|e^{-sT}\|_{\mathcal B}\le e^{-sm}$, hence
\begin{equation}\label{eq:sylvester_inverse_bound_sec3}
\|X\|_{\mathcal B}
\le \int_0^\infty \|e^{-sS}\|_{\mathcal B}\,\|Y\|_{\mathcal B}\,\|e^{-sT}\|_{\mathcal B}\,ds
\le \frac{1}{\lambda_{\min}(S)+\lambda_{\min}(T)}\|Y\|_{\mathcal B}
\le \frac{1}{2m}\|Y\|_{\mathcal B}.
\end{equation}
These estimates are crucial for the control of the conjugated Dirac-type operator $\widetilde D(A)$ on the fixed domain $H^1(M,\Sigma_gM)$. To this end, we also require the following boundedness for $W^{1,p}$ endomorphisms.
\begin{lemma}\label{lem:H1_multiplier_W1p_sec3}
Let $F \in W^{1,p}\bigl(M,\End(\Sigma_g M)\bigr)$ with $p > n$. Then each~$\psi \in H^1(M,\Sigma_g M)$ is mapped to~$F\psi\in H^1(M,\Sigma_g M)$. Moreover, the map
\begin{align}
H^1(M,\Sigma_g M) \to H^1(M,\Sigma_g M),\qquad \psi \longmapsto F\psi,
\end{align}
is bounded.
\end{lemma}

\begin{proof}
Denote by~$\snabla$ the spin connection on~$\Sigma_g M$, and by~$\nabla^g$ the induced covariant derivative on~$\End(\Sigma_g M)$.

For~$\psi \in H^1$ we estimate the $H^1$-norm of $F\psi$:
\begin{align}
\|F\psi\|_{H^1}^2 = \|F\psi\|_{L^2}^2 + \|\snabla(F\psi)\|_{L^2}^2 .
\end{align}
Using the pointwise bound~$|F\psi| \le \|F\|_{L^\infty}|\psi|$ and the Sobolev embedding~$W^{1,p} \hookrightarrow L^\infty$, since~$p>n\geq 2$, we obtain
\begin{align}
\|F\psi\|_{L^2} \le \|F\|_{L^p} \|\psi\|_{L^q}.
\end{align}
where~$\frac{1}{2}=\frac{1}{p}+\frac{1}{q}$, i.e.~$q=\frac{2p}{p-2} > 2$.

By the Leibniz rule,
\begin{align}
\snabla(F\psi) = (\nabla^g F)\psi + F(\snabla\psi).
\end{align}
Hence
\begin{align}
\|\snabla(F\psi)\|_{L^2}
\leq& \|(\nabla^g F)\psi\|_{L^2} + \|F(\snabla\psi)\|_{L^2} \\
\leq&  \|\nabla^g F\|_{L^p} \|\psi\|_{L^q}
        +  \|F\|_{L^\infty} \|\snabla\psi\|_{L^2}.
\end{align}
Let~$C_1$ be the Sobolev constant for the embedding~$W^{1,p}\hookrightarrow L^\infty$ on~$M$, and~$C_2$ be the constant for the embedding~$H^1\hookrightarrow L^q$ with~$q$ as above.
Then
\begin{align}
    \|F\psi\|_{H^1}
    \leq& C_2\|F\|_{W^{1,p}} \|\psi\|_{H^1} + C_1\|F\|_{W^{1,p}}\|\psi\|_{H^1}
    =(C_1+C_2)\|F\|_{W^{1,p}}\|\psi\|_{H^1}.
\end{align}
Hence the map $\psi \mapsto F\psi$ is bounded from $H^1$ to $H^1$.
\end{proof}
That is,~$W^{1,p}\bigl(M,\End(\Sigma_g M)\bigr)$ embeds into $\mathcal{B}(H^1(M,\Sigma_g M),H^1(M,\Sigma_g M))$, the space of bounded linear maps, with
\begin{align}
 \|F\|_{\mathcal{B}(H^1,H^1)} \leq (C_1+ C_2)\|F\|_{W^{1,p}}.
\end{align}
The next theorem records the basic properties of the weighted operators and of the conjugation map.

\begin{thm}\label{thm:isometric_weighted_dirac_realparam}
Let $I\subset\mathbb R$ be an interval, and let~$I\ni t\mapsto A_t\in\mathcal P_{\Lambda_1,\Lambda_2}$
be continuous with respect to~$\|\cdot\|_{\mathcal A}$.

\begin{enumerate}
\item[(1)] The map
\begin{align}
U_{A_t}:\mathcal H_{A_t}\longrightarrow L^2(M, \Sigma_gM),
\qquad
U_{A_t}\psi:=A_t^{1/2}\psi,\notag
\end{align}
is a linear \emph{isometric isomorphism}, and
\begin{align}
\widetilde D(A_t) = U_{A_t}\,B(A_t)\,U_{A_t}^{-1}
\quad\text{acts on }L^2(M, \Sigma_gM)\text{ with domain }H^1(M, \Sigma_gM).\notag
\end{align}
That is, the following diagram commutes:
\begin{center}
\begin{tikzcd}[column sep=large, row sep=large]
H^1(M, \Sigma_gM) \arrow[r, "\tilde{D}(A_t)"] \arrow[d, "A_t^{-1/2}"'] & L^2(M, \Sigma_gM) \\
\bigl(H^1(M, \Sigma_gM), (\cdot, \cdot)_{A_t}\bigr) \arrow[r, "A_t^{-1}\D"'] & \mathcal H_{A_t} \arrow[u, "A_t^{1/2}"'].
\end{tikzcd}
\end{center}
\item[(2)] The operator $\widetilde D(A_t)$ is self-adjoint on $H^1(M, \Sigma_gM)$, hence closed and densely defined. Consequently $B(A_t)$ is self-adjoint on $\mathcal H_{A_t}$. Moreover, $B(A_t)$ and $\widetilde D(A_t)$ are \emph{isospectral}.
\item[(3)] The map
\begin{align}
I\longrightarrow\mathcal B\bigl(H^1(M, \Sigma_gM),L^2(M, \Sigma_gM)\bigr),
\qquad
t\longmapsto \widetilde D(A_t),\notag
\end{align}
is continuous, with respect to the operator norm on $\mathcal B(H^1,L^2)$.
\end{enumerate}

\noindent
In particular, the map
\begin{align}
\widetilde D:\mathcal P_{\Lambda_1,\Lambda_2}\longrightarrow\mathcal B\bigl(H^1(M, \Sigma_gM),L^2(M, \Sigma_gM)\bigr),
\qquad
A\longmapsto \widetilde D(A),\notag
\end{align}
is continuous.
\end{thm}

\begin{proof}
We write $\nabla^{g}$ for the covariant derivative on~$\End(\Sigma_gM)$ induced by the spin connection on~$\Sigma_gM$.

\smallskip
\noindent\textbf{Step 1.} We investigate the continuity properties of $A_t^{\frac{1}{2}}$ and $A_t^{-\frac{1}{2}}$ with respect to $t$ in $\mathcal P_{\Lambda_1,\Lambda_2}$. Fix $t\in I$. By the definition of $\mathcal P_{\Lambda_1,\Lambda_2}$, $A_t$ is fiberwisely self-adjoint and uniformly positive definite. Hence the positive square root is well-defined, and we set
\begin{align}
U_t:=A_t^{1/2},\qquad V_t:=A_t^{-1/2}:=U_t^{-1}.\notag
\end{align}
From the eigenvalue bounds we obtain
\begin{equation}\label{eq:St_bounds_sec3}
\|U_t\|_{L^\infty}\le \sqrt{\Lambda_2},\quad
\|V_t\|_{L^\infty}\le \frac{1}{\sqrt{\Lambda_1}},
\quad\forall t\in I.
\end{equation}

First, we show that $U_A,V_A\in W^{1,p}(\End(\Sigma_g M))$.
Applying $\nabla^g$ to $U_t^2=A_t$ yields the Sylvester-type equation a.e.
\begin{equation}\label{eq:grad_sqrt_equation_W1p_sec3}
U_t(\nabla^g U_t)+(\nabla^g U_t)U_t=\nabla^g A_t.
\end{equation}
Applying \eqref{eq:sylvester_inverse_bound_sec3} to \eqref{eq:grad_sqrt_equation_W1p_sec3} with $S=T=U_t$ and $m=\sqrt{\Lambda_1}$
gives the estimate $|\nabla^g U_t|\le \frac{1}{2\sqrt{\Lambda_1}}|\nabla^gA_t|$, hence
\begin{equation}\label{eq:grad_St_bound_W1p_sec3}
\|\nabla^g U_t\|_{L^p}\le \frac{1}{2\sqrt{\Lambda_1}}\|\nabla^g A_t\|_{L^p}.
\end{equation}
Differentiating $V_tU_t=\Id$ gives
\begin{align}
\nabla^g V_t=-V_t(\nabla^g U_t)V_t,\notag
\end{align}
and using \eqref{eq:St_bounds_sec3} and \eqref{eq:grad_St_bound_W1p_sec3} we obtain
\begin{equation}\label{eq:grad_Q_bound_W1p_sec3}
\|\nabla^g V_t\|_{L^p}
\le \|V_t\|_{L^\infty}^2\,\|\nabla^g U_t\|_{L^p}
\le \frac{1}{2\Lambda_1^{3/2}}\|\nabla^g A_t\|_{L^p}.
\end{equation}
Thus $U_t,V_t\in W^{1,p}\bigl(M,\End(\Sigma_gM)\bigr)$ and
\begin{align}
\|U_t\|_{\mathcal A}+\|V_t\|_{\mathcal A}\le C(\Lambda_1,\Lambda_2)\bigl(1+\|\nabla^gA_t\|_{L^p}\bigr).\notag
\end{align}

\smallskip
Second, we prove both $U_t$ and $V_t$ are continuous in $t$.
For $s,t\in I$ we have
\begin{align}
U_t^2-U_s^2=A_t-A_s
\quad\Longrightarrow\quad
U_t(U_t-U_s)+(U_t-U_s)U_s=A_t-A_s.\notag
\end{align}
Let $X:=U_t-U_s$. Then $X$ solves $\mathcal L_{U_t,U_s}(X)=A_t-A_s$.
Using \eqref{eq:sylvester_inverse_bound_sec3} with $m=\sqrt{\Lambda_1}$ yields
\begin{equation}\label{eq:St_Linfty_Lipschitz_W1p_sec3}
\|U_t-U_s\|_{L^\infty}\le \frac{1}{2\sqrt{\Lambda_1}}\|A_t-A_s\|_{L^\infty}.
\end{equation}
Evaluating~\eqref{eq:grad_sqrt_equation_W1p_sec3} at $t$ and $s$ and subtracting yields, with
$Y:=\nabla^gU_t-\nabla^gU_s$,
\begin{align}
\mathcal{L}_{U_t, U_s}(Y)
= U_tY+YU_s
= \nabla^g(A_t-A_s)
-(U_t-U_s)\nabla^gU_s-(\nabla^gU_t)(U_t-U_s).\notag
\end{align}
It follows that
\begin{align}\label{eq:grad_St_Lp_est_W1p_sec3}
\|\nabla^gU_t-\nabla^gU_s\|_{L^p}
&\le \frac{1}{2\sqrt{\Lambda_1}}\Bigl(
\|\nabla^g(A_t-A_s)\|_{L^p}
+\|U_t-U_s\|_{L^\infty}\bigl(\|\nabla^gU_t\|_{L^p}+\|\nabla^gU_s\|_{L^p}\bigr)
\Bigr).
\end{align}
Using \eqref{eq:St_Linfty_Lipschitz_W1p_sec3}, \eqref{eq:grad_St_bound_W1p_sec3} and the $\mathcal A$--continuity of $t\mapsto A_t$,
we infer that $t\mapsto U_t$ is continuous in $\mathcal A$.

Moreover,
\begin{align}
V_t-V_s= U_t^{-1}-U_s^{-1}=U_t^{-1}(U_s-U_t)U_s^{-1},\notag
\end{align}
hence by \eqref{eq:St_bounds_sec3} and \eqref{eq:St_Linfty_Lipschitz_W1p_sec3},
\begin{equation}\label{eq:Qt_Linfty_cont_W1p_sec3}
\|V_t-V_s\|_{L^\infty}
\le \frac{1}{2\Lambda_1^{3/2}}\|A_t-A_s\|_{L^\infty}.
\end{equation}
Finally, from $\nabla^gV_t=-V_t(\nabla^gU_t)V_t$ we get
\begin{align}
\nabla^gV_t-\nabla^gV_s
&=
-(V_t-V_s)(\nabla^gU_t)V_t
-V_s(\nabla^gU_t-\nabla^gU_s)V_t
-V_s(\nabla^gU_s)(V_t-V_s).\notag
\end{align}
Using \eqref{eq:grad_St_bound_W1p_sec3}, \eqref{eq:grad_St_Lp_est_W1p_sec3} and
\eqref{eq:Qt_Linfty_cont_W1p_sec3}, we conclude that $t\mapsto V_t$ is continuous in $\mathcal A$.

\smallskip
\noindent\textbf{Step 2.} We prove (1).
Fix $A\in\mathcal P_{\Lambda_1,\Lambda_2}$. For any $\psi,\phi\in L^2(M, \Sigma_gM)$,
\begin{align}
(U_A\psi,U_A\phi)_{L^2}
=\int_M\langle A^{1/2}\psi, A^{1/2}\phi\rangle\,dv_g
=\int_M\langle A\psi,\phi\rangle\,dv_g
=(\psi,\phi)_A,\notag
\end{align}
so $U_A$ is an isometry from $\mathcal H_A$ onto~$L^2(M, \Sigma_gM)$. Since $A^{1/2}$ is invertible fiberwise,
$U_A$ is bijective with inverse $U_A^{-1}\phi=A^{-1/2}\phi$.
Next, for $\psi\in H^1(M, \Sigma_gM)$,
\begin{align}
U_A\,B(A)\,U_A^{-1}\psi
=
A^{1/2}\,A^{-1}\D\,(A^{-1/2}\psi)
=
A^{-1/2}\D\,(A^{-1/2}\psi)
=
\widetilde D(A)\psi,\notag
\end{align}
this verifies the claim.

\smallskip
\noindent\textbf{Step 3.} We prove (2).
Fix $A\in\mathcal P_{\Lambda_1,\Lambda_2}$ and set $V:=A^{-1/2}\in W^{1,p}(M,\End(\Sigma_gM))$.
By Lemma~\ref{lem:H1_multiplier_W1p_sec3}, action by~$V$ and $V^{-1}=A^{1/2}$ is bounded on $H^1(M, \Sigma_gM)$.
The operator $\widetilde D(A)=V\D V=A^{-1/2}\D A^{-1/2}$ is symmetric on $H^1(M, \Sigma_gM)$:
for $\phi,\psi\in H^1(M, \Sigma_gM)$,
\begin{align}
\langle \widetilde D(A)\phi,\psi\rangle_{L^2}
=\langle \D(V\phi),V\psi\rangle_{L^2}
=\langle V\phi,\D(V\psi)\rangle_{L^2}
=\langle \phi, V\D(V\psi)\rangle_{L^2}
=\langle \phi,\widetilde D(A)\psi\rangle_{L^2}.\notag
\end{align}
Let $\varphi\in\Dom(\widetilde D(A)^*)$ and set $\widetilde{\varphi}:=\widetilde D(A)^*\varphi\in L^2(M, \Sigma_gM)$.
Then for all $\psi\in H^1(M, \Sigma_gM)$,
\begin{align}
\langle \widetilde D(A)\psi,\varphi\rangle_{L^2}=\langle \psi,\widetilde{\varphi}\rangle_{L^2}.\notag
\end{align}
Writing $\chi:=V\psi\in H^1(M, \Sigma_gM)$, we obtain
\begin{align}
\langle \D\chi,V\varphi\rangle_{L^2}=\langle \chi,V^{-1}\widetilde{\varphi}\rangle_{L^2},\qquad\forall \chi\in H^1(M, \Sigma_gM).\notag
\end{align}
Therefore, $V\varphi \in \Dom(\D^*) = \Dom(\D) = H^1(M, \Sigma_gM)$, and $\D^*(V\varphi) = V^{-1}\widetilde{\varphi} \in L^2(M, \Sigma_gM)$. Since~$V^{-1}$ is bounded on $H^1(M, \Sigma_gM)$, it follows that $\varphi = V^{-1}(V\varphi) \in H^1(M, \Sigma_gM)$. Consequently,~$\Dom(\widetilde D(A)^*) \subseteq H^1(M, \Sigma_gM) = \Dom(\widetilde D(A))$, hence $\widetilde D(A)$ is self-adjoint. Through the isometric conjugation $\widetilde D(A) = U_A B(A) U_A^{-1}$, we deduce that $B(A)$ is self-adjoint on $\mathcal H_A$. This completes the proof of assertion (2).

\smallskip
\noindent\textbf{Step 4.} We prove (3).
Fix $s,t\in I$ and set $V_s:=A_s^{-1/2}$ and $V_t:=A_t^{-1/2}$.
For $\psi\in H^1(M, \Sigma_gM)$,
\begin{align}
\bigl(\widetilde D(A_t)-\widetilde D(A_s)\bigr)\psi
=
V_t\D(V_t\psi)-V_s\D(V_s\psi)
=
(V_t-V_s)\D(V_t\psi)+V_s\D\bigl((V_t-V_s)\psi\bigr).\notag
\end{align}
Given that the operator $\D:H^1\to L^2$ is bounded and that weighting by $W^{1,p}$ endomorphisms is bounded on both $H^1$ and $L^2$, see Lemma~\ref{lem:H1_multiplier_W1p_sec3}, we obtain
\begin{equation}\label{eq:DtDs_estimate_W1p_sec3}
\|\widetilde D(A_t)-\widetilde D(A_s)\|_{\mathcal B(H^1,L^2)}
\le C\,\|V_t-V_s\|_{\mathcal A}\,\bigl(\|V_t\|_{\mathcal A}+\|V_s\|_{\mathcal A}\bigr),
\end{equation}
for a constant $C=C(M,g,p)$.
By Step~1, $t\mapsto V_t$ is continuous in $\mathcal A$, hence the right-hand side tends to $0$ as $t\to s$.
\end{proof}

Theorem~\ref{thm:isometric_weighted_dirac_realparam} verifies the conditions of
Definition~\ref{def:typeA-family}. After checking the remaining requirements, we obtain the following

\begin{cor}\label{cor:typeA_discrete_family}
The map
\begin{align}
\widetilde D:\mathcal P_{\Lambda_1,\Lambda_2}\longrightarrow\mathcal C\bigl(L^2(M, \Sigma_gM)\bigr),\qquad
A\longmapsto \widetilde D(A):=A^{-1/2}\D\,A^{-1/2},\notag
\end{align}
is a discrete self-adjoint family of type \emph{(A)} with common domain
\begin{align}
\Dom\bigl(\widetilde D(A)\bigr)=H^1(M, \Sigma_gM).\notag
\end{align}
\end{cor}

\begin{proof}
Fix $A\in\mathcal P_{\Lambda_1,\Lambda_2}$ and set $V:=A^{-1/2}\in W^{1,p}(M,\End(\Sigma_gM))$.
By Theorem~\ref{thm:isometric_weighted_dirac_realparam}, $\widetilde D(A)$ is self-adjoint on $L^2(M, \Sigma_gM)$ with
$\Dom(\widetilde D(A))=H^1(M, \Sigma_gM)$ and the map
\begin{align}
A\longmapsto \widetilde D(A)\in\mathcal B\bigl(H^1(M, \Sigma_gM),L^2(M, \Sigma_gM)\bigr)\notag
\end{align}
is continuous in $\|\cdot\|_{\mathcal A}$. It remains to verify the graph-norm condition.

We equip $H^1(M, \Sigma_gM)$ with the standard norm $\|\cdot\|_{H^1}$, since the action of a spinor bundle endomorphism with $W^{1,p}$ regularity on $H^1$ is bounded and $\D:H^1(M, \Sigma_gM)\to L^2(M, \Sigma_gM)$ is bounded, we have $\widetilde D(A)\in\mathcal B(H^1,L^2)$. Hence, for any $\varphi \in H^1(M, \Sigma_gM)$,
\begin{align}
\|\varphi\|_{L^2}+\|\widetilde D(A)\varphi\|_{L^2}\le C\|\varphi\|_{H^1}.\notag
\end{align}
For the converse inequality, we use the standard elliptic estimate for the Dirac operator:
there exists~$C>0$ such that for all $\psi\in H^1(M, \Sigma_gM)$,
\begin{equation}\label{eq:Dirac_elliptic_est_sec3}
\|\psi\|_{H^1}\le C\bigl(\|\psi\|_{L^2}+\|\D\psi\|_{L^2}\bigr).
\end{equation}
For any $\varphi\in H^1(M, \Sigma_gM)$, set $\psi:=V\varphi\in H^1(M, \Sigma_gM)$. Then
\begin{align}
\D\psi=\D(V\varphi)=V^{-1}\widetilde D(A)\varphi,\notag
\end{align}
hence by \eqref{eq:Dirac_elliptic_est_sec3} and \eqref{eq:St_bounds_sec3},
\begin{align}
\|\psi\|_{H^1}\le C\bigl(\|\varphi\|_{L^2}+\|\widetilde D(A)\varphi\|_{L^2}\bigr).\notag
\end{align}
Since $\varphi=V^{-1}\psi$ and $V^{-1}=A^{1/2}\in W^{1,p}(M,\End(\Sigma_gM))$, Lemma~\ref{lem:H1_multiplier_W1p_sec3} gives
\begin{align}
\|\varphi\|_{H^1}\le C'\|\psi\|_{H^1}
\le C''\bigl(\|\varphi\|_{L^2}+\|\widetilde D(A)\varphi\|_{L^2}\bigr).\notag
\end{align}
Thus the graph norm of $\widetilde D(A)$ is equivalent to $\|\cdot\|_{H^1}$ on $H^1(M, \Sigma_gM)$.

Finally, we show that $\widetilde{D}(A)$ has a compact resolvent.

For any $\varphi \in L^2\left(M, \Sigma_g M\right)$, let $\psi:= (\widetilde{D}(A)-i)^{-1} \varphi$, where~$i\in \C$ is the imaginary unit.
Since $\widetilde{D}(A)$ is self-adjoint, taking the imaginary part of $\langle(\widetilde{D}(A)- i) \psi, \psi\rangle_{L^2}=\langle\varphi, \psi\rangle_{L^2}$ yields
\begin{align}
\|\psi\|_{L^2}^2 \leq\|\varphi\|_{L^2}\|\psi\|_{L^2} \quad \Longrightarrow \quad\|\psi\|_{L^2} \leq\|\varphi\|_{L^2}.
\end{align}
Furthermore, the identity $\widetilde{D}(A) \psi=\varphi+i \psi$ implies that $\|\widetilde{D}(A) \psi\|_{L^2} \leq\|\varphi\|_{L^2}+ \|\psi\|_{L^2} \leq 2\|\varphi\|_{L^2}$. Combining this with the graph-norm equivalence established above, there exists a constant $C>0$ such that
\begin{align}
\|\psi\|_{H^1} \leq C\left(\|\psi\|_{L^2}+\|\widetilde{D}(A) \psi\|_{L^2}\right) \leq 3 C\|\varphi\|_{L^2}.
\end{align}
This proves that the resolvent $(\widetilde{D}(A)-i)^{-1}: L^2 \rightarrow H^1$ is a bounded linear operator. By the Rellich-Kondrachov theorem on the compact manifold $M$, the embedding $\iota$ : $H^1\left(M, \Sigma_g M\right) \hookrightarrow L^2\left(M, \Sigma_g M\right)$ is compact. Consequently, the composition
\begin{align}
(\widetilde{D}(A)-i)^{-1}=\iota \circ(\widetilde{D}(A)-i)^{-1}: L^2 \longrightarrow L^2
\end{align}
is a compact operator. Thus, each $\widetilde{D}(A)$ has a compact resolvent, and the family is discrete.
\end{proof}

Building on the preceding discussion, we can establish the following theorem, showing that the full spectrum depends continuously with respect to the $\arsinh$-metric on the weight in the $W^{1,p}$ topology. Theorem \ref{thm: C0_CONT} follows as a consequence.
\begin{thm}\label{thm: C0 main thm}
The map
\begin{align}
\mathfrak s:\bigl(\mathcal P_{\Lambda_1,\Lambda_2},\|\cdot\|_{\mathcal A}\bigr)\longrightarrow (\mathfrak{Mon},d_a),
\qquad
A\longmapsto \mathfrak s^{A},\notag
\end{align}
is continuous.
\end{thm}

\begin{proof}
Fix $A_0\in\mathcal P_{\Lambda_1,\Lambda_2}$ and set
\begin{align}
\widetilde D(A):=A^{-1/2}\D\,A^{-1/2},
\qquad
\Dom\bigl(\widetilde D(A)\bigr)=H^1(M, \Sigma_gM).\notag
\end{align}
By Corollary~\ref{cor:typeA_discrete_family}, the map $A\mapsto \widetilde D(A)$ is a discrete self-adjoint
family of type (A) on the Hilbert space~$L^2(M, \Sigma_gM)$.

We utilize Theorem~\ref{thm: key expanding} on this family at $A_0$.
Given $\varepsilon>0$, there exists a $\|\cdot\|_{\mathcal A}$-neighborhood $U$ of~$A_0$ in
$\mathcal P_{\Lambda_1,\Lambda_2}$ such that for each $A\in U$ there is an integer~$k(A)\in\bZ$ with
\begin{equation}\label{eq:shifted_da_est_W1p_sec3}
\forall j\in\bZ:\quad d_a\!\bigl(\mathfrak s^{A_0}(j),\,\mathfrak s^{A}(j+k(A))\bigr)<\varepsilon .
\end{equation}
We show that one can take $k(A)=0$ for each $A\in U$.

To this end we first note that for $\psi\in H^1(M, \Sigma_gM)$,
\begin{align}
\widetilde D(A)\psi=0
\Longleftrightarrow
\D(A^{-1/2}\psi)=0
\Longleftrightarrow
A^{-1/2}\psi\in \ker(\D),\notag
\end{align}
so
\begin{align}
\ker\bigl(\widetilde D(A)\bigr)=A^{1/2}\ker(\D),
\qquad
\dim \ker\bigl(\widetilde D(A)\bigr)=\dim\ker(\D),
\quad \forall A\in\mathcal P_{\Lambda_1,\Lambda_2}.\notag
\end{align}
In particular, the multiplicity of the eigenvalue $0$ is constant on $\mathcal P_{\Lambda_1,\Lambda_2}$.

Let $m:=\dim\ker(\D)$, by the indexing convention, $0$ occurs with multiplicity $m$ and
\begin{align}
\mathfrak s^{A}(0)=\cdots=\mathfrak s^{A}(m-1)=0,\quad
\mathfrak s^{A}(-1)<0,\quad \mathfrak s^{A}(m)>0,
\quad \forall A\in\mathcal P_{\Lambda_1,\Lambda_2}.\notag
\end{align}

Next we choose $\varepsilon_\ast>0$ such that
\begin{align}
0<\varepsilon_\ast<\frac12\min\Bigl\{\,\arsinh\bigl(-\mathfrak s^{A_0}(-1)\bigr),\ \arsinh\bigl(\mathfrak s^{A_0}(m)\bigr)\Bigr\}.
\notag
\end{align}
Set $\varepsilon_0:=\min\{\varepsilon,\varepsilon_\ast\}$ and shrink $U$ so that
\eqref{eq:shifted_da_est_W1p_sec3} holds with $\varepsilon_0$ in place of $\varepsilon$.

We now claim that $k(A)=0$ for all $A\in U$.
Assume by contradiction that $k(A)\neq 0$ for some $A\in U$.

If $k(A)\ge 1$, then taking $j=-1$ in \eqref{eq:shifted_da_est_W1p_sec3} gives
\begin{align}
d_a\!\bigl(\mathfrak s^{A_0}(-1),\,\mathfrak s^{A}(-1+k(A))\bigr)<\varepsilon_0.\notag
\end{align}
Since $-1+k(A)\ge 0$ and $\mathfrak s^{A}$ is nondecreasing, we have
\begin{align}
\mathfrak s^{A}(-1+k(A))\ge \mathfrak s^{A}(0)=0.\notag
\end{align}
On the other hand, $\mathfrak s^{A_0}(-1)<0$, hence
\begin{align}
d_a\!\bigl(\mathfrak s^{A_0}(-1),\,\mathfrak s^{A}(-1+k(A))\bigr)
&\ge \arsinh\!\bigl(-\mathfrak s^{A_0}(-1)\bigr)
\ge 2\varepsilon_\ast\ge 2\varepsilon_0,\notag
\end{align}
a contradiction.

If $k(A)\le -1$, then taking $j=m$ in \eqref{eq:shifted_da_est_W1p_sec3} gives
\begin{align}
d_a\!\bigl(\mathfrak s^{A_0}(m),\,\mathfrak s^{A}(m+k(A))\bigr)<\varepsilon_0.\notag
\end{align}
Since $m+k(A)\le m-1$ and $\mathfrak s^{A}$ is nondecreasing, we have
\begin{align}
\mathfrak s^{A}(m+k(A))\le \mathfrak s^{A}(m-1)=0.\notag
\end{align}
On the other hand, $\mathfrak s^{A_0}(m)>0$, hence
\begin{align}
d_a\!\bigl(\mathfrak s^{A_0}(m),\,\mathfrak s^{A}(m+k(A))\bigr)
&\ge \arsinh\!\bigl(\mathfrak s^{A_0}(m)\bigr)
\ge 2\varepsilon_\ast\ge 2\varepsilon_0,\notag
\end{align}
a contradiction.

Thus $k(A)=0$ for all $A\in U$.
Consequently, we have
\begin{align}
d_a\!\bigl(\mathfrak s^{A_0}(j),\,\mathfrak s^{A}(j)\bigr)<\varepsilon, \qquad\forall A\in U.\notag
\end{align}
This proves continuity of $\mathfrak s$ at $A_0$ in $(\mathfrak{Mon},d_a)$.
\end{proof}

\section[C1 weight]{\texorpdfstring{The Lipschitz Continuity of Weighted Spectra with $C^1$ Weight}{The Lipschitz Continuity of Weighted Spectra with C1 Weight}}\label{sec:C1weight}

In this section, we consider a $C^1$ family $(A_t)_{t\in I}\subset\mathcal P_{\Lambda_1,\Lambda_2}$ on an interval $I\subset\mathbb R$.
Let $\mathfrak s_t\in\mathfrak{Mon}$ be the spectral tuple associated with the weighted eigenvalue problem
$\D\psi=\lambda A_t\psi$.
The goal of this section is to prove Theorem~\ref{thm:C1_CONT}, namely the local Lipschitz continuity
for $t\mapsto\mathfrak s_t$ in the $\arsinh$-metric.

A convenient way to study the weighted problem is to pass to the conjugated family
\begin{align}
\widetilde{D}_t:=A_t^{-1/2}\D A_t^{-1/2},
\end{align}
which has the same eigenvalues as the equation~$\D\psi=\lambda A_t\psi$.
We therefore start with the~$C^1$-regularity of the~$A_t^{\pm1/2}$.
\begin{lemma}\label{lemma:sqrt-C1}
Let $p > n$ and assume \eqref{hyp:H1}, i.e.\ the family
$t \mapsto A_t \in \mathcal{P}^{p}_{\Lambda_1,\Lambda_2}$
is of class $C^1$ with respect to the $W^{1,p}$-topology on
$W^{1,p}(M, \End(\Sigma_g M))$.
Set $S_t := A_t^{1/2}$ and $Q_t := A_t^{-1/2}$.
Then the maps
\begin{align}
    t \longmapsto S_t
    \qquad\text{and}\qquad
    t \longmapsto Q_t
\end{align}
are both of class $C^1$ as maps
$I \to W^{1,p}(M, \End(\Sigma_g M))$.
\end{lemma}

\begin{proof}
We write $\nabla^{g}$ for the covariant derivative on~$\End(\Sigma_gM)$
induced by the spin connection on~$\Sigma_gM$.

Recall that for fiberwise self-adjoint endomorphisms $S, T \geq m\,\Id$
with $m > 0$, the Sylvester operator
$\mathcal{L}_{S,T}(X) := SX + XT$
is invertible with unique solution
$\mathcal{L}_{S,T}^{-1}(Y) = \int_0^\infty e^{-\tau S}Ye^{-\tau T}\,\dd\tau$
and operator-norm bound
$\|\mathcal{L}_{S,T}^{-1}(Y)\|_{\mathcal B} \leq \frac{1}{2m}\|Y\|_{\mathcal B}$.

\smallskip
We begin by identifying the candidate derivative.
Differentiating the identity $S_t^2 = A_t$ formally gives the Sylvester equation
\begin{equation}\label{eq:sylvester-dot}
    \mathcal{L}_{S_t,S_t}(\dot S_t)
    = S_t\dot S_t + \dot S_t S_t = \dot A_t,
\end{equation}
whose unique solution is the integral formula in Lemma~\ref{lem: Sylvester solution},
with $L^\infty$-bound
$\|\dot S_t\|_{L^\infty}
\leq \frac{1}{2\sqrt{\Lambda_1}}\|\dot A_t\|_{L^\infty}$.
Applying $\nabla^{g}$ to~\eqref{eq:sylvester-dot} yields
\begin{align}
    \mathcal{L}_{S_t,S_t}(\nabla^{g}\dot S_t)
    = \nabla^{g}\dot A_t
      - (\nabla^{g} S_t)\dot S_t
      - \dot S_t(\nabla^{g} S_t).
\end{align}
The right-hand side lies in $L^p$: the first term by \eqref{hyp:H1}, and the
remaining terms since $\nabla^{g} S_t \in L^p$ (by estimate~\eqref{eq:grad_St_bound_W1p_sec3}) and
$\dot S_t \in L^\infty$.
Hence $\dot S_t \in W^{1,p}(M, \End(\Sigma_g M))$.

\smallskip
It remains to verify that $\dot S_t$ is indeed the $W^{1,p}$-derivative of $S_t$.
Set $D_h := \frac{S_{t+h}-S_t}{h}$.
From $S_{t+h}^2 - S_t^2 = A_{t+h} - A_t$ one reads off
\begin{equation}\label{eq:Dh-sylvester}
    \mathcal{L}_{S_{t+h},S_t}(D_h) = \frac{A_{t+h}-A_t}{h},
\end{equation}
so the remainder $R_h := D_h - \dot S_t$ satisfies
\begin{equation}\label{eq:Rh}
    \mathcal{L}_{S_{t+h},S_t}(R_h) = E_h + F_h,
    \qquad
    E_h := \frac{A_{t+h}-A_t}{h} - \dot A_t,
    \quad
    F_h := (S_t - S_{t+h})\dot S_t.
\end{equation}
The $L^\infty$-bound on $\mathcal{L}_{S_{t+h},S_t}^{-1}$, combined with~$\|E_h\|_{L^\infty} \to 0$, since $W^{1,p} \hookrightarrow L^\infty$ and~$\frac{A_{t+h}-A_t}{h} \to \dot A_t$ in~$W^{1,p}$ by \eqref{hyp:H1} and
\begin{align}
    \|F_h\|_{L^\infty}
    \leq \|S_t - S_{t+h}\|_{L^\infty}\|\dot S_t\|_{L^\infty}
    \leq \frac{\|A_t - A_{t+h}\|_{L^\infty}}{2\sqrt{\Lambda_1}}
         \|\dot S_t\|_{L^\infty}
    \longrightarrow 0.
\end{align}
We have
\begin{equation}\label{eq:Rh-Linfty}
    \|R_h\|_{L^\infty} \longrightarrow 0 \quad\text{as } h \to 0.
\end{equation}

To upgrade to $W^{1,p}$-convergence, we apply $\nabla^{g}$ to~\eqref{eq:Rh}.
The Leibniz rule gives
\begin{equation}\label{eq:nabla-Rh}
    \mathcal{L}_{S_{t+h},S_t}(\nabla^{g} R_h)
    = \nabla^{g} E_h + \nabla^{g} F_h
      - (\nabla^{g} S_{t+h})R_h - R_h(\nabla^{g} S_t),
\end{equation}
and the $L^p$-bound on $\mathcal{L}_{S_{t+h},S_t}^{-1}$ reduces the
claim $\|\nabla^{g} R_h\|_{L^p} \to 0$ to showing that each term on the
right of~\eqref{eq:nabla-Rh} tends to zero in $L^p$.

The term $\|\nabla^{g} E_h\|_{L^p} \to 0$ is immediate from hypothesis \eqref{hyp:H1}.
For $\nabla^{g} F_h$, the Leibniz rule gives
\begin{align}
    \nabla^{g} F_h
    = (\nabla^{g} S_t - \nabla^{g} S_{t+h})\dot S_t
      + (S_t - S_{t+h})\,\nabla^{g}\dot S_t;
\end{align}
the first term tends to zero in $L^p$ because
$\|\nabla^{g} S_t - \nabla^{g} S_{t+h}\|_{L^p} \to 0$, continuity of $t \mapsto S_t$ in $W^{1,p}$, estimate~\eqref{eq:grad_St_Lp_est_W1p_sec3}
and $\dot S_t \in L^\infty$;
the second because $\|S_t - S_{t+h}\|_{L^\infty} \to 0$
while $\nabla^{g}\dot S_t \in L^p$.
For the last two terms in~\eqref{eq:nabla-Rh}, we note that
estimate~\eqref{eq:St_Linfty_Lipschitz_W1p_sec3} gives a uniform bound
$\sup_{|h| \leq 1}\|\nabla^{g} S_{t+h}\|_{L^p} < \infty$,
so
\begin{align}
    \|(\nabla^{g} S_{t+h})R_h\|_{L^p}
    \leq \|\nabla^{g} S_{t+h}\|_{L^p}\,\|R_h\|_{L^\infty}
    \longrightarrow 0,
\end{align}
and the symmetric estimate $\|R_h(\nabla^{g} S_t)\|_{L^p}
\leq \|\nabla^{g} S_t\|_{L^p}\|R_h\|_{L^\infty} \to 0$
follows from~\eqref{eq:Rh-Linfty}.
Therefore~$\|R_h\|_{W^{1,p}} \to 0$, confirming
$\frac{S_{t+h}-S_t}{h} \to \dot S_t$ in $W^{1,p}$.

\smallskip
Continuity of $t \mapsto \dot S_t$ in $W^{1,p}$ follows from the
integral representation.
Decomposing the following equation
\begin{align}
    \dot S_t - \dot S_s
    = \int_0^\infty
      \bigl(
          e^{-\tau S_t}\dot A_t e^{-\tau S_t}
         -e^{-\tau S_s}\dot A_s e^{-\tau S_s}
      \bigr)\dd\tau
\end{align}
into three parts, one involving~$\dot A_t - \dot A_s$, and two involving~$e^{-\tau S_t}-e^{-\tau S_s}$—each
tends to zero in~$W^{1,p}$: the first by \eqref{hyp:H1}, the others by continuity
of~$t \mapsto S_t$ in~$W^{1,p}$ and smoothness of the matrix exponential.
Since every term is dominated by~$e^{-2\tau\sqrt{\Lambda_1}}$, yields~$\|\dot S_t - \dot S_s\|_{W^{1,p}} \to 0$,
so~$t \mapsto S_t$ is~$C^1$.

\smallskip
Finally, the formula $\dot Q_t = -Q_t\dot S_t Q_t$, obtained by
differentiating $Q_tS_t = \Id$, shows that $\dot Q_t \in W^{1,p}$
since $W^{1,p}(M, \End(\Sigma_g M))$ is a Banach algebra for $p > n$.
Continuity of $t \mapsto \dot Q_t$ in $W^{1,p}$ follows from
that of $t \mapsto Q_t$ and $t \mapsto \dot S_t$,
completing the proof that $t \mapsto Q_t = A_t^{-1/2}$ is $C^1$.
\end{proof}

Once the~$C^1$-dependence of~$A_t^{\pm1/2}$ is available, the conjugated operators~$\widetilde{D}_t$ form a $C^1$ family from~$H^1(M,\Sigma_gM)$ to~$L^2(M,\Sigma_gM)$. This provides the analytic input needed to study eigenvalues through their projections and to construct corresponding~$C^1$ local frames.
\begin{thm}\label{thm:multiple_cluster_projector}
Assume \eqref{hyp:H1}.
Fix $t_0\in I$, and let $\lambda_0\in\operatorname{Spect}(\widetilde{D}_{t_0})$
be an eigenvalue of multiplicity~$m$.
Choose a smooth positively oriented Jordan curve $\Gamma\subset\C$
which encloses $\lambda_0$ and no other points of~$\operatorname{Spect}(\widetilde{D}_{t_0})$.
For $t$ near $t_0$, define the projection
\begin{align}
P_t:=\frac{1}{2\pi i}\oint_{\Gamma}(z-\widetilde{D}_t)^{-1}\,\dd z.
\end{align}
Then, after shrinking $I$ around $t_0$ if necessary, the following hold.
\begin{enumerate}
\item The map $t\mapsto P_t$ is of class $C^1$ as a map
\begin{align}
I \longrightarrow
\mathcal B\bigl(L^2(M,\Sigma_gM),H^1(M,\Sigma_gM)\bigr),
\end{align}
with derivative
\begin{align}\label{eq:Pt-deriv}
\frac{\dd}{\dd t}P_t
=
\frac{1}{2\pi i}\oint_{\Gamma}
(z-\widetilde{D}_t)^{-1}\,\dot{\widetilde{D}}_t\,(z-\widetilde{D}_t)^{-1}\,\dd z.
\end{align}

\item If $u_1^0,\dots,u_m^0$ is any $L^2$-orthonormal basis of $\Ran(P_{t_0})$,
then there exist $C^1$ maps
\begin{align}
u_1,\dots,u_m\colon U\longrightarrow H^1(M,\Sigma_gM)
\end{align}
on some neighbourhood $U\subset I$ of $t_0$ such that
\begin{align}
u_i(t_0)=u_i^0,\qquad
\langle u_i(t),u_j(t)\rangle_{L^2}=\delta_{ij},
\qquad
\Ran(P_t)=\Span\{u_1(t),\dots,u_m(t)\}.
\end{align}
Consequently, setting
\begin{align}
\phi_i(t):=A_t^{-1/2}u_i(t),\qquad i=1,\dots,m,
\end{align}
one obtains a $C^1$ family of $A_t$-orthonormal frames of the weighted spectral subspace
\begin{align}
E_t:=A_t^{-1/2}\Ran(P_t)\subset H^1(M,\Sigma_gM).
\end{align}
\end{enumerate}
\end{thm}

\begin{proof}
First, by Lemma~\ref{lemma:sqrt-C1}, the map
$t\mapsto A_t^{-1/2}$ is of class $C^1$ as a map
$I\to W^{1,p}(M,\End(\Sigma_gM))$,
and the map
\begin{align}
t\longmapsto \widetilde{D}_t
\in\mathcal B\bigl(H^1(M,\Sigma_gM),L^2(M,\Sigma_gM)\bigr)
\end{align}
is continuously differentiable.

For $z\in \Gamma$, write
\begin{align}
R_t(z):=(z-\widetilde{D}_t)^{-1}.
\end{align}
Since $\Gamma\subset \C\setminus \operatorname{Spect}(\widetilde{D}_{t_0})$ is compact and
$z\mapsto R_{t_0}(z)$ is continuous as an
$\mathcal B(L^2,H^1)$-valued map, we may define
\begin{align}
M_{\Gamma}:=\sup_{z\in \Gamma}\|R_{t_0}(z)\|_{\mathcal B(L^2,H^1)}<\infty.
\end{align}
Because $t\mapsto \widetilde{D}_t$ is continuous in
$\mathcal B(H^1,L^2)$, after shrinking the interval we may assume that
\begin{align}
\|\widetilde{D}_t-\widetilde{D}_{t_0}\|_{\mathcal B(H^1,L^2)},
\le \frac{1}{2M_{\Gamma}}
\qquad\text{for all }t\text{ near }t_0.
\end{align}
Hence, for every $z\in \Gamma$,
\begin{align}
\|(\widetilde{D}_t-\widetilde{D}_{t_0})R_{t_0}(z)\|_{\mathcal B(L^2)}
\le
\|\widetilde{D}_t-\widetilde{D}_{t_0}\|_{\mathcal B(H^1,L^2)}
\|R_{t_0}(z)\|_{\mathcal B(L^2,H^1)}
\le \frac12.
\end{align}
Therefore,
\begin{align}
z-\widetilde{D}_t
=
\bigl[\Id-(\widetilde{D}_t-\widetilde{D}_{t_0})R_{t_0}(z)\bigr]
(z-\widetilde{D}_{t_0})
\end{align}
is invertible by the Neumann series, and thus
\begin{align}
\Gamma\subset \C\setminus \operatorname{Spect}(\widetilde{D}_t)
\end{align}
for all $t$ near $t_0$.

Since each $\widetilde{D}_t$ is self-adjoint with compact resolvent,
$P_t$ is the spectral projection associated with an eigenvalue of multiplicity $m$ inside $\Gamma$; in particular $P_t$ is an orthogonal projection.

We next prove the $C^1$-dependence of $P_t$.
The resolvent identity yields
\begin{align}
R_t(z)-R_s(z)
=
R_t(z)(\widetilde{D}_t-\widetilde{D}_s)R_s(z).
\end{align}
Since $t\mapsto \widetilde{D}_t$ is of class $C^1$ in
$\mathcal B(H^1,L^2)$, it follows that
$t\mapsto R_t(z)$ is of class $C^1$ in
$\mathcal B(L^2,H^1)$, and
\begin{align}
\dot{R}_t(z)=R_t(z)\dot{\widetilde{D}}_tR_t(z).
\end{align}
Integrating over $\Gamma$, we obtain
\begin{align}
P_t=\frac{1}{2\pi i}\oint_{\Gamma}R_t(z)\,\dd z
\end{align}
and
\begin{align}
\dot{P}_t
=
\frac{1}{2\pi i}\oint_{\Gamma}\dot{R}_t(z)\,\dd z
=
\frac{1}{2\pi i}\oint_{\Gamma}
(z-\widetilde{D}_t)^{-1}\dot{\widetilde{D}}_t(z-\widetilde{D}_t)^{-1}\,\dd z.
\end{align}
This proves (1).

To prove that $\rank(P_t)$ is constant, note that $P_t\to P_{t_0}$ in
$\mathcal B(L^2)$. After shrinking the interval further we may assume
\begin{align}
\|P_t-P_{t_0}\|_{\mathcal B(L^2)}<1,
\qquad\text{for all }t\text{ near }t_0.
\end{align}
If $u\in \Ran(P_{t_0})$ and $P_tu=0$, then
\begin{align}
u=P_{t_0}u-P_tu=(P_{t_0}-P_t)u,
\end{align}
hence
\begin{align}
\|u\|_{L^2}\le \|P_t-P_{t_0}\|_{\mathcal B(L^2)}\|u\|_{L^2}<\|u\|_{L^2},
\end{align}
a contradiction. Therefore
\begin{align}
P_t|_{\Ran(P_{t_0})}:\Ran(P_{t_0})\longrightarrow \Ran(P_t)
\end{align}
is injective, and thus $\rank(P_{t_0})\le \rank(P_t)$.
Interchanging $t$ and $t_0$ gives the reverse inequality, so
\begin{align}
\rank(P_t)=\rank(P_{t_0})=m.
\end{align}

Finally, choose an $L^2$-orthonormal basis $u_1^0,\dots,u_m^0$ of $\Ran(P_{t_0})$, and set
\begin{align}
v_i(t):=P_tu_i^0,\qquad i=1,\dots,m.
\end{align}
By the injectivity just proved, $v_1(t),\dots,v_m(t)$ are linearly independent for $t$ near $t_0$.
Define the Gram matrix
\begin{align}
G(t):=
\bigl(\langle v_i(t),v_j(t)\rangle_{L^2}\bigr)_{1\le i,j\le m}.
\end{align}
Then $G(t)$ is a positive-definite matrix depending $C^1$ on $t$, and so does $G(t)^{-1/2}$.

Define
\begin{align}
u_i(t):=\sum_{k=1}^m v_k(t)\bigl(G(t)^{-1/2}\bigr)_{ki},
\qquad i=1,\dots,m.
\end{align}
Then
\begin{align}
\langle u_i(t),u_j(t)\rangle_{L^2}=\delta_{ij},
\qquad
\Ran(P_t)=\Span\{u_1(t),\dots,u_m(t)\},
\end{align}
and each $u_i(t)$ depends $C^1$ on $t$ as an $H^1$-valued map.
This proves the first part of (2).

Now define
\begin{align}
\phi_i(t):=A_t^{-1/2}u_i(t).
\end{align}
By the  $C^1$-regularity of $t\mapsto A_t^{-1/2}$, each $\phi_i(t)$ is $C^1$ in $H^1$.
Moreover,
\begin{align}
(\phi_i(t),\phi_j(t))_{A_t}
=
\int_M \langle A_t\phi_i(t),\phi_j(t)\rangle\,d\mathrm{vol}_g
=
\int_M \langle A_t^{1/2}\phi_i(t),A_t^{1/2}\phi_j(t)\rangle\,d\mathrm{vol}_g
=
\langle u_i(t),u_j(t)\rangle_{L^2}
=
\delta_{ij}.
\end{align}
Since $\Ran(P_t)=\Span\{u_1(t),\dots,u_m(t)\}$, it follows that
\begin{align}
E_t=A_t^{-1/2}\Ran(P_t)=\Span\{\phi_1(t),\dots,\phi_m(t)\}.
\end{align}
The proof is complete.
\end{proof}

In the present setting, Theorem~\ref{thm:multiple_cluster_projector} identifies the $C^1$ projection of an eigenvalue of multiplicity $m$ and the corresponding local $C^1$ frames. The following remarks explain how this relates to the classical simple-eigenvalue situation, and then to the difficulties caused by crossings.
\begin{rmk}
For a simple eigenvalue, one may use Lyapunov--Schmidt reduction
together with the implicit function theorem to obtain a local $C^1$
branch of eigenpairs $(\lambda_n,\varphi_n)$; see e.g.
\cite[Section~2]{Ambrosetti2007NonlinerAna}.
In our setting, however, simplicity cannot be taken for granted.
Already for the classical Dirac operator~$\D$,
eigenvalues are in general not simple, see e.g. \cite{Dahl2003Dirac}.
The weighted operator~$B_t = A_t^{-1}\D$ inherits this
obstruction: its eigenvalues have multiplicity greater than one in general, and no choice of admissible weight~$A_t \in \mathcal{P}_{\Lambda_1,\Lambda_2}$ can reduce them to simple ones. It is therefore more natural to formulate the perturbation theory at the level of eigenvalues of higher multiplicity. This is precisely the role of Theorem~\ref{thm:multiple_cluster_projector}, which yields~$C^1$
projections and, consequently,~$C^1$ local frames of the
corresponding eigenspaces.
\end{rmk}
The previous remark describes the favourable simple--situation. In contrast, when eigenvalues of higher multiplicity give rise to crossings, see Figure~\ref{fig:branch-projected-2d}, the $C^1$-regularity furnished by Theorem~\ref{thm:multiple_cluster_projector} should no longer be interpreted as a canonical $C^1$--regularity of the globally ordered eigenvalue branches~$\mathfrak{s}_t$.
\begin{rmk}
The $C^1$-regularity furnished by Theorem~\ref{thm:multiple_cluster_projector} pertains to the projection of an eigenvalue with higher multiplicity and to local frames of the corresponding spectral subspace. It should not be confused with $C^1$-regularity of the globally ordered eigenvalue map
\begin{align}
t\longmapsto \mathfrak s_t(j).
\end{align}
Indeed, even in finite dimensions, a crossing of smooth eigenvalue branches may produce a corner after sorting. For example,
\begin{align}
H(t)=
\begin{pmatrix}
t & 0\\
0 & 2-t
\end{pmatrix}
\end{align}
has smooth eigenvalue branches $t$ and $2-t$, whereas the ordered eigenvalues are $1-|t-1|$ and $1+|t-1|$, which are not $C^1$ at $t=1$. Thus, in the presence of crossings, one should regard the $C^1$ objects as the local spectral projections and the labelled branches they determine at each eigenvalue, rather than the globally sorted branches themselves. The weighted Hellmann--Feynman identity below is therefore formulated for a locally chosen $C^1$ eigenpair branch.
\end{rmk}
Before turning to variation formulas, we record the precise spectral regularity furnished by the previous analysis. The natural $C^1$ object is not an individual eigenvalue branch a priori, but rather the spectral projection associated with an eigenvalue. More precisely, we consider the conjugated family
\begin{align}
\widetilde{D}_t &= A_t^{-1/2}\D A_t^{-1/2},\\
B_t &= A_t^{-1}\D,\\
U_t &= A_t^{1/2},
\end{align}
so that
\begin{align}
B_t = U_t^{-1}\widetilde{D}_t U_t.
\end{align}
Under \eqref{hyp:H1}, the maps $t\mapsto A_t^{\pm 1/2}$ are of class $C^1$, and hence
\begin{align}
t\longmapsto \widetilde{D}_t
\in \mathcal B\bigl(H^1(M,\Sigma_gM),L^2(M,\Sigma_gM)\bigr)
\end{align}
is a $C^1$ family. Consequently, if $\Gamma\subset\C$ is a fixed positively oriented Jordan curve which encloses an eigenvalue and does not meet the rest of the spectrum, then the associated projection
\begin{align}
\widetilde{P}_t:=\frac{1}{2\pi i}\oint_{\Gamma}(z-\widetilde{D}_t)^{-1}\,\dd z
\end{align}
depends~$C^1$ on~$t$. Equivalently, the corresponding projection in the weighted picture,
\begin{align}
P_t:=U_t^{-1}\widetilde{P}_tU_t,
\end{align}
is also of class $C^1$. Thus the preceding results provide $C^1$-regularity for weighted eigenspaces and for local frames, but they do not by themselves produce a canonical $C^1$ labeling of individual eigenvalues. The variational identity below should therefore be understood as applying once a $C^1$ eigenpair branch has been chosen.

\medskip
\noindent
We next derive the weighted Hellmann--Feynman variational identity for $C^1$ eigenpair branches of
\begin{align}
\D\varphi(t)=\lambda(t)\,A_t\varphi(t),
\end{align}
under the normalization $\int_M\langle A_t\varphi(t),\varphi(t)\rangle\,dv_g=1$.

\begin{thm}\label{thm:WHF_prelim}
Let $I\subset\R$ be an interval and let $t\mapsto A_t\in \mathcal P_{\Lambda_1,\Lambda_2}$ be of class $C^1$.
Assume there exists a $C^1$ eigenpair branch
$t\mapsto (\lambda(t),\varphi(t))$ with $\varphi(t)\in H^1(M, \Sigma_gM)$ such that
\begin{align}\label{eq:eigenpair_branch_prelim}
\D\varphi(t)=\lambda(t)\,A_t\varphi(t),
\qquad
\int_M\Abracket{A_t\varphi(t),\varphi(t)}\,\dv_g=1,
\qquad \forall t\in I.
\end{align}
Then for all $t\in I$,
\begin{equation}\label{eq:WHF}
\lambda'(t)=-\lambda(t)\int_M\Abracket{(\p_tA_t)\varphi(t),\varphi(t)}\,\dv_g.
\end{equation}
\end{thm}

\begin{proof}
Fix $t\in I$ and write $\lambda=\lambda(t)$ and $\varphi=\varphi(t)$.
Differentiate the normalization in \eqref{eq:eigenpair_branch_prelim}. By the product rule,
\begin{equation}\label{eq:dnorm}
0=\int_M\Abracket{\p_tA_t\,\varphi,\varphi}\,\dv_g
 +\int_M\Abracket{A_t\,\p_t\varphi,\varphi}\,\dv_g
 +\int_M\Abracket{A_t\,\varphi,\p_t\varphi}\,\dv_g.
\end{equation}
Since $A_t$ is fiberwise self-adjoint,
$\Abracket{A_t\varphi,\p_t\varphi}=\Abracket{A_t\p_t\varphi,\varphi}$, hence
\begin{equation}\label{eq:dnorm-realpart}
 0=\int_M\Abracket{\p_tA_t\,\varphi,\varphi}\,\dv_g
 +2\int_M\Abracket{A_t\,\p_t\varphi,\varphi}\,\dv_g.
\end{equation}

Differentiate the eigenvalue equation $\D\varphi=\lambda\,A_t\varphi$:
\begin{align}
\D(\p_t\varphi)
=\lambda' A_t\varphi+\lambda(\p_tA_t)\varphi+\lambda A_t(\p_t\varphi).
\end{align}
Pair with $\varphi$ in $L^2$ to obtain
{ \small
\begin{equation}\label{eq:ip}
 \int_M\Abracket{\D(\p_t\varphi),\varphi}\,\dv_g
=\lambda'\int_M\Abracket{A_t\varphi,\varphi}\,\dv_g
 +\lambda\int_M\Abracket{\p_tA_t\,\varphi,\varphi}\,\dv_g
 +\lambda\int_M\Abracket{A_t(\p_t\varphi),\varphi}\,\dv_g.
\end{equation}
}
Using the $L^2$ self-adjointness of $\D$ and the eigenvalue equation again,
\begin{align}
 \int_M\Abracket{\D(\p_t\varphi),\varphi}\,\dv_g
=\int_M\Abracket{\p_t\varphi,\D\varphi}\,\dv_g
=\lambda\int_M\Abracket{\p_t\varphi,A_t\varphi}\,\dv_g
=\lambda\int_M\Abracket{A_t(\p_t\varphi),\varphi}\,\dv_g.
\end{align}
Substituting this into \eqref{eq:ip} and canceling the common term yields
\begin{align}
0=\lambda'\int_M\Abracket{A_t\varphi,\varphi}\,\dv_g
 +\lambda\int_M\Abracket{\p_tA_t\,\varphi,\varphi}\,\dv_g.
\end{align}
By $\int_M\Abracket{A_t\varphi,\varphi}\dv_g=1$, this is \eqref{eq:WHF}.
\end{proof}

As a first consequence, we obtain an $\arsinh$-Lipschitz bound along any eigenvalue.

\begin{cor}\label{cor:asinhLip_branch_prelim}
In the setting of Theorem~\ref{thm:WHF_prelim}, let
\begin{align}\label{eq:LJ_def_prelim}
L_I:=\frac{C_1}{\Lambda_1}\sup_{\tau\in I}\|\p_\tau A_\tau\|_{W^{1,p}},
\end{align}
where $C_1$ is the Sobolev constant for the embedding of $W^{1,p}$ into $L^{\infty}$.

Then for all $a,b\in I$,
\begin{align}\label{eq:asinhLip_branch_est_prelim}
|\arsinh(\lambda(a))-\arsinh(\lambda(b))|\le L_I\,|a-b|.
\end{align}
\end{cor}

\begin{proof}
From the normalization in \eqref{eq:eigenpair_branch_prelim} and $A_t\ge \Lambda_1\Id$,
\begin{align}
1=\int_M\Abracket{A_t\varphi(t),\varphi(t)}\,\dv_g \ \ge\ \Lambda_1\int_M|\varphi(t)|^2\,\dv_g,
\end{align}
hence $\|\varphi(t)\|_{L^2}^2\le \Lambda_1^{-1}$ for all $t\in I$.
Therefore,
\begin{align}
\Bigl|\int_M\Abracket{(\p_tA_t)\varphi(t),\varphi(t)}\,\dv_g\Bigr|
\le \|\p_tA_t\|_{L^\infty}\,\|\varphi(t)\|_{L^2}^2
\le \frac{\|\p_tA_t\|_{L^\infty}}{\Lambda_1}
\le L_I.
\end{align}
Combining with \eqref{eq:WHF} yields
\begin{align}
|\lambda'(t)|\le |\lambda(t)|\,L_I,
\qquad\text{for all }t\in I.
\end{align}
Consequently,
\begin{align}
\Bigl|\frac{\dd}{\dd t}\arsinh(\lambda(t))\Bigr|
=\Bigl|\frac{\lambda'(t)}{\sqrt{1+\lambda(t)^2}}\Bigr|
\le L_I\,\frac{|\lambda(t)|}{\sqrt{1+\lambda(t)^2}}
\le L_I.
\end{align}
Integrating over the interval with endpoints $a$ and $b$ gives \eqref{eq:asinhLip_branch_est_prelim}.
\end{proof}

To pass from branchwise estimates to the ordered spectrum, we use a stability statement for sorting.

\begin{lemma}\label{lem:sorting-stability}
Let $a_1 \le \cdots \le a_N$ and $b_1 \le \cdots \le b_N$ be two nondecreasing sequences of real numbers.
If there exists a permutation $\sigma \in S_N$ and $\delta > 0$ such that
\begin{align}
|a_i - b_{\sigma(i)}| \le \delta, \qquad \forall i.
\end{align}
Then
\begin{align}
|a_i - b_i| \le \delta, \qquad \forall i.
\end{align}
In particular, define~$\alpha_i:=\arsinh(a_i)$ and~$\beta_i:=\arsinh(b_i)$.
If there exists a permutation~$\sigma \in S_N$ and~$\delta_1 > 0$ such that
\begin{align}
|\alpha_i - \beta_{\sigma(i)}| \le \delta_1 \qquad \forall i,
\end{align}
then
\begin{align}
|\alpha_i - \beta_i| \le \delta_1 \qquad \forall i.
\end{align}
\end{lemma}

\begin{proof}
We prove the first statement by contradiction.
Suppose there exists~$k$ with $a_k > b_k + \delta$ and set~$x := b_k + \delta$.
Since~$b_1 \le \cdots \le b_k$, we have~$b_i \le b_k$ for all~$i \le k$, hence~$b_i + \delta \le x$.
By hypothesis, for each~$i \le k$ we have $a_{\sigma^{-1}(i)} \le b_i + \delta \le x$.
Thus at least~$k$ elements of the sequence $(a_i)$ are less than or equal to~$x$.
But $a_k > x$ implies that in the sorted sequence $a_1 \le \cdots \le a_N$ at most $k-1$ terms can be less than or equal to~$x$,
a contradiction. The case~$b_k > a_k + \delta$ is symmetric.

The $\arsinh$--statement follows since $\arsinh$ is strictly increasing and the sequences
$(\alpha_i)$ and $(\beta_i)$ are nondecreasing.
\end{proof}

\medskip
\noindent
We now combine the weighted Hellmann--Feynman identity with spectral localization and sorting to control the spectrum.

\begin{thm}\label{thm:MonLip_prelim}
Assume \eqref{hyp:H1}. Let $\mathfrak s_t\in\mathfrak{Mon}$ be the spectral tuple associated
with \eqref{eq:WD}. Then
\begin{align}\label{eq:MonLip}
d_a(\mathfrak s_t,\mathfrak s_s)\le L_I\,|t-s|,
\qquad \forall\,s,t\in I,
\end{align}
where $L_I$ denotes the constant introduced in \eqref{eq:LJ_def_prelim}.

In particular, $t\mapsto \mathfrak s_t$ is locally Lipschitz as a map into $(\mathfrak{Mon},d_a)$.
\end{thm}

\begin{proof}
The weighted eigenvalue equation \eqref{eq:WD} is equivalent to the spectral equation $B_t\psi=\lambda\psi$. The indexing used to define $\mathfrak s_t$ does not change
with $t$ on $I$. This is already proved in Theorem~\ref{thm: C0 main thm} by showing that the
multiplicity of the eigenvalue $0$ is constant along the family. In particular, $\mathfrak s_t$ is a
well-defined element of $\mathfrak{Mon}$ for every $t\in I$.

Let $t_0\in I$ and assume $\lambda(t_0)$ is an eigenvalue of $B_{t_0}$. Then there exists a~$C^1$ eigenpair branch~$t\mapsto(\lambda(t),\varphi(t))$ on a neighborhood of $t_0$ such that
\begin{align}
\D\varphi(t)=\lambda(t)A_t\varphi(t),
\qquad
\int_M\langle A_t\varphi(t),\varphi(t)\rangle\,dv_g=1.
\end{align}
By the weighted Hellmann--Feynman identity and Corollary~\ref{cor:asinhLip_branch_prelim},
the function $\arsinh(\lambda(t))$ is Lipschitz on the branch, with
\begin{align}\label{eq:branch_asinh_Lip}
|\arsinh(\lambda(t))-\arsinh(\lambda(s))|\le L_I|t-s|,
\end{align}
for all $s,t$ in the interval where the branch is defined.

We now pass from a branch to the ordered spectrum, and we explain how crossings are handled.
Fix $t_0\in I$ and $R>0$. Since $B_{t_0}$ has compact resolvent, $\sigma(B_{t_0})\cap[-R,R]$ consists of finitely many eigenvalues, counted with multiplicity, where~$\sigma(B_{t_0})$ denotes the spectrum of~$B_{t_0}$ as a set.
Choose a smooth positively oriented contour~$\Gamma$ enclosing $\sigma(B_{t_0})\cap[-R,R]$
and no other eigenvalues, and define the projector
\begin{align}
P_t:=\frac{1}{2\pi i}\oint_\Gamma (B_t-z)^{-1}\,dz.
\end{align}
For $t$ close to $t_0$ the contour stays in the resolvent set, $\mathrm{rank}\,P_t$ is constant,
and $\mathrm{Ran}(P_t)$ is a finite-dimensional spectral subspace whose spectrum coincides with
$\sigma(B_t)\cap[-R,R]$. The restriction of $B_t$ to $\mathrm{Ran}(P_t)$ gives a finite-dimensional
self-adjoint family, hence its eigenvalues can be described by continuous eigenvalue functions.
Whenever such a local eigenvalue function is~$C^1$ in~$t$, it satisfies the estimate~\eqref{eq:branch_asinh_Lip}.

At an eigenvalue crossing the individual branches may fail to be
differentiable, and the labels may interchange.
This does not affect the estimate, because the finite-dimensional reduction still controls the eigenvalues in the window as a list counted with multiplicity.
After relabelling through the crossing, one can match the eigenvalues at two nearby $s,t$ in the window~$[-R,R]$ so that each matched pair satisfies
\begin{align}\label{eq:window_matching}
|\arsinh(\lambda_s)-\arsinh(\lambda_t)|\le L_I|t-s|.
\end{align}

For arbitrary $s,t\in I$ within the fixed window~$[-R,R]$, applying Lemma~\ref{lem:sorting-stability} gives the
same bound for the nondecreasingly ordered eigenvalues in~$[-R,R]$.

Finally, fix $s,t\in I$ and $j\in\mathbb Z$. Select $R$ sufficiently large such that $\mathfrak s_s(j)$ and $\mathfrak s_t(j)$ lie in~$[-R,R]$. The window estimate consequently yields
\begin{align}
|\arsinh(\mathfrak s_t(j))-\arsinh(\mathfrak s_s(j))|\le L_I|t-s|.
\end{align}
Taking the supremum over $j\in\mathbb Z$ gives \eqref{eq:MonLip}.
\end{proof}

To help the reader better visualize the discussion in this section, we make a final remark. The key observation is that branches which overlap or intersect in a two-dimensional projection remain separated in the corresponding three-dimensional representation. More precisely, once local labels are chosen near a multiple eigenvalue, one may work with the resulting $C^1$ branches obtained from the corresponding spectral projections; however, after these local labels are forgotten and the eigenvalues are reordered globally, the resulting sorted branches need not remain $C^1$ at a crossing.

Figures~\ref{fig:branch-labelled-3d} and~\ref{fig:branch-projected-2d} illustrate this point concretely. Figure~\ref{fig:branch-labelled-3d} displays the labelled branches as distinct curves in the coordinate space $(j,t,\lambda)$, making the $C^1$ regularity of each individual branch directly visible. By contrast, Figure~\ref{fig:branch-projected-2d} shows what happens after the branch labels are suppressed and the eigenvalues are reordered: in the projected two-dimensional picture, the corner produced by the sorting map at a crossing becomes clearly visible.

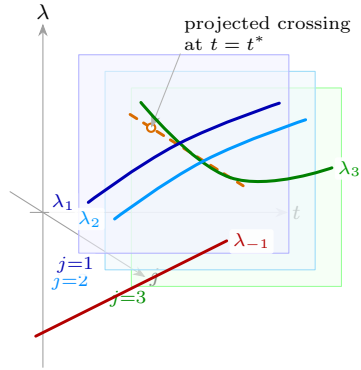
\begin{figure}[ht]
\centering
\begin{tikzpicture}[
  scale=0.92,
  >=Stealth,
  line cap=round,
  line join=round,
  albl/.style={font=\scriptsize, inner sep=1.5pt},
  tlbl/.style={
    font=\tiny,
    inner sep=1.3pt,
    fill=white,
    fill opacity=0.9,
    text opacity=1,
    rounded corners=1pt
  },
]

\draw[->, gray!75] \Pthd{0}{-0.2}{0} -- \Pthd{0}{3.85}{0} node[albl, right, text=black] {$t$};
\draw[->, gray!75] \Pthd{0}{0}{-3.2} -- \Pthd{0}{0}{3.9} node[albl, above, text=black] {$\lambda$};
\draw[->, gray!75] \Pthd{-1.25}{0}{0} -- \Pthd{3.9}{0}{0} node[albl, right, text=black] {$j$};


\fill[green!6, opacity=0.55]
  \Pthd{3}{0.15}{-0.5} -- \Pthd{3}{3.45}{-0.5} -- \Pthd{3}{3.45}{3.60} -- \Pthd{3}{0.15}{3.60} -- cycle;
\draw[green!40, thin]
  \Pthd{3}{0.15}{-0.5} -- \Pthd{3}{3.45}{-0.5} -- \Pthd{3}{3.45}{3.60} -- \Pthd{3}{0.15}{3.60} -- cycle;
\node[tlbl, text=green!55!black, anchor=north] at \Pthd{3}{0.10}{-0.5} {$j{=}3$};


\fill[cyan!5, opacity=0.6]
  \Pthd{2}{0.15}{-0.5} -- \Pthd{2}{3.45}{-0.5} -- \Pthd{2}{3.45}{3.60} -- \Pthd{2}{0.15}{3.60} -- cycle;
\draw[cyan!35, thin]
  \Pthd{2}{0.15}{-0.5} -- \Pthd{2}{3.45}{-0.5} -- \Pthd{2}{3.45}{3.60} -- \Pthd{2}{0.15}{3.60} -- cycle;

\node[tlbl, text=blue!45!cyan, anchor=north east] at \Pthd{2}{-0.05}{-0.5} {$j{=}2$};

\fill[blue!5, opacity=0.55]
  \Pthd{1}{0.15}{-0.5} -- \Pthd{1}{3.45}{-0.5} -- \Pthd{1}{3.45}{3.60} -- \Pthd{1}{0.15}{3.60} -- cycle;
\draw[blue!30, thin]
  \Pthd{1}{0.15}{-0.5} -- \Pthd{1}{3.45}{-0.5} -- \Pthd{1}{3.45}{3.60} -- \Pthd{1}{0.15}{3.60} -- cycle;
\node[tlbl, text=blue!70!black, anchor=north] at \Pthd{1}{0.10}{-0.5} {$j{=}1$};


\draw[orange!85!black, line width=1.1pt, dashed]
  \Pthd{-0.8}{1.70}{1.75} -- \Pthd{3.5}{1.70}{1.75};

\node[circle, fill=white, draw=orange!85!black, inner sep=1.1pt, line width=0.8pt]
  at \Pthd{0}{1.70}{1.75} {};

\draw[<-, gray!60, thin]
  \Pthd{0}{1.70}{1.75} -- \Pthd{-0.95}{2.55}{2.95}
  node[tlbl, text=black, anchor=south west, align=left]
  {projected crossing\\at $t=t^*$};


\draw[line width=1.15pt, green!50!black]
  plot[smooth, tension=0.72] coordinates {
    \Pthd{3}{0.30}{3.30} \Pthd{3}{1.70}{1.75} \Pthd{3}{3.30}{1.95}
  };
\node[tlbl, text=green!60!black, anchor=west]
  at \Pthd{3.04}{3.32}{1.95} {$\lambda_3$};

\draw[line width=1.15pt, blue!40!cyan]
  plot[smooth, tension=0.75] coordinates {
    \Pthd{2}{0.30}{0.55} \Pthd{2}{1.70}{1.75} \Pthd{2}{3.30}{2.60}
  };
\node[tlbl, text=blue!40!cyan, anchor=east] at \Pthd{2}{0.15}{0.55} {$\lambda_2$};

\draw[line width=1.15pt, blue!70!black]
  plot[smooth, tension=0.75] coordinates {
    \Pthd{1}{0.30}{0.55} \Pthd{1}{1.70}{1.75} \Pthd{1}{3.30}{2.60}
  };
\node[tlbl, text=blue!70!black, anchor=east] at \Pthd{1}{0.15}{0.55} {$\lambda_1$};

\draw[line width=1.1pt, red!70!black]
  plot[smooth, tension=0.8] coordinates {
    \Pthd{-1}{0.30}{-2.9} \Pthd{-1}{3.30}{-0.93}
  };
\node[tlbl, text=red!70!black, anchor=west]
  at \Pthd{-0.96}{3.32}{-0.93} {$\lambda_{-1}$};
\end{tikzpicture}
\caption{Three-dimensional branch-labelled view of the local weighted spectrum.}
\label{fig:branch-labelled-3d}
\end{figure}
Figure~\ref{fig:branch-labelled-3d} presents the branch-labelled three-dimensional picture. The curves $\lambda_1$ and $\lambda_2$ lie on different planes $j=1$ and $j=2$, so they remain distinct in three dimensions even though their projections onto the $(t,\lambda)$-plane coincide. The projection of the branch $\lambda_3$ onto the $(t,\lambda)$-plane passes through the same point at $t=t_0$, highlighted by the orange dashed line, while the negative index branch $\lambda_{-1}$ stays strictly below $\lambda=0$ throughout the interval. Overall, this figure emphasizes that branch crossing in projection does not destroy the regularity of the branch curves.

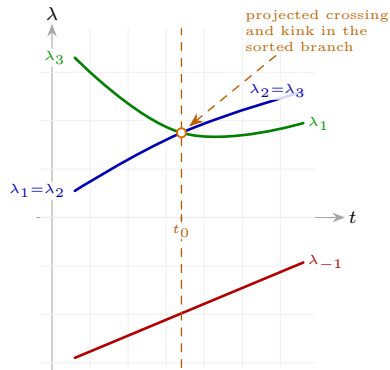
\begin{figure}[ht]
\centering
\begin{tikzpicture}[
    scale=0.92,
    >=Stealth,
    line cap=round,
    line join=round,
    albl/.style={font=\tiny, inner sep=1.2pt},
    tlbl/.style={
        font=\fontsize{5}{6}\selectfont,
        inner sep=1pt,
        fill=white,
        fill opacity=0.85,
        text opacity=1,
        rounded corners=0.8pt
    },
]

\draw[->, gray!70] \Ptwo{-0.2}{0} -- \Ptwo{3.85}{0}
  node[albl, right, text=black] {$t$};
\draw[->, gray!70] \Ptwo{0}{-3.2} -- \Ptwo{0}{4.0}
  node[albl, above, text=black] {$\lambda$};

\foreach \tt in {0.5,1.0,1.5,2.0,2.5,3.0} {
    \draw[gray!12, line width=0.2pt]
      \Ptwo{\tt}{-3.1} -- \Ptwo{\tt}{3.9};
}
\foreach \lam in {-3,-2,-1,0,1,2,3} {
    \draw[gray!12, line width=0.2pt]
      \Ptwo{-0.15}{\lam} -- \Ptwo{3.45}{\lam};
}


\draw[line width=1pt, red!70!black]
  plot[smooth, tension=0.80] coordinates {
    \Ptwo{0.30}{-2.90}
    \Ptwo{3.30}{-0.93}
  };
\node[tlbl, text=red!70!black, anchor=west]
  at \Ptwo{3.32}{-0.93} {$\lambda_{-1}$};

\draw[line width=1pt, blue!70!black]
  plot[smooth, tension=0.75] coordinates {
    \Ptwo{0.30}{0.55}
    \Ptwo{1.70}{1.75}
    \Ptwo{3.30}{2.60}
  };
\node[tlbl, text=blue!70!black, anchor=east]
  at \Ptwo{0.22}{0.55} {$\lambda_1{=}\lambda_2$};
\node[tlbl, text=blue!70!black, anchor=south west]
  at \Ptwo{2.55}{2.42} {$\lambda_2{=}\lambda_3$};

\draw[line width=1pt, green!50!black]
  plot[smooth, tension=0.72] coordinates {
    \Ptwo{0.30}{3.30}
    \Ptwo{1.70}{1.75}
    \Ptwo{3.30}{1.95}
  };
\node[tlbl, text=green!60!black, anchor=east,
      font=\fontsize{5}{6}\selectfont]
  at \Ptwo{0.22}{3.30} {$\lambda_3$};
\node[tlbl, text=green!60!black, anchor=west]
  at \Ptwo{3.32}{1.95} {$\lambda_1$};

\draw[orange!70!black, dashed, line width=0.5pt]
  \Ptwo{1.70}{-3.1} -- \Ptwo{1.70}{3.9};

\node[circle, fill=white, draw=orange!85!black,
      inner sep=1.2pt, line width=0.7pt]
  at \Ptwo{1.70}{1.75} {};

\draw[orange!70!black, line width=0.5pt]
  \Ptwo{1.70}{0.06} -- \Ptwo{1.70}{-0.06};
\node[tlbl, text=orange!70!black, anchor=north]
  at \Ptwo{1.70}{-0.08} {$t_0$};

\draw[->, orange!75!black, dashed, line width=0.55pt]
  \Ptwo{3.10}{3.55} -- \Ptwo{1.80}{1.88};
\node[tlbl, text=orange!75!black, anchor=south west, align=left]
  at \Ptwo{2.50}{3.38}
  {projected crossing\\and kink in the\\sorted branch};

\end{tikzpicture}
\caption{Projection onto the $(t,\lambda)$-plane and the apparent kink after sorting.}
\label{fig:branch-projected-2d}
\end{figure}

Figure~\ref{fig:branch-projected-2d} shows the projection of each
labelled branch in Figure~\ref{fig:branch-labelled-3d} onto the
$(t,\lambda)$-plane.
The red curve corresponds to the projection of~$\lambda_{-1}$, which
remains strictly on the side~$\lambda<0$ throughout and does not cross
the zero level; the dark blue curve corresponds to the coincident
projection of~$\lambda_1$ and~$\lambda_2$; the green curve corresponds
to the projection of~$\lambda_3$, and intersects the former at $t=t_0$.
The figure illustrates that, even though the local labelled branches
obtained from Theorem~\ref{thm:multiple_cluster_projector} are each~$C^1$ on their respective domains, the sorted spectral branches
obtained after reordering may develop a kink at a crossing point and
thereby lose~$C^1$ regularity.
In particular, at the crossing time $t_0$ one has
$\lambda_1(t_0)=\lambda_2(t_0)=\lambda_3(t_0)$,
so that the three branches merge at a single eigenvalue which
acquires triple multiplicity at $t_0$; it is precisely this sudden
increase in multiplicity that causes the sorted branches to lose
differentiability here.

The labels on the left and right sides of the figure correspond to the
sorting at $t<t_0$ and $t>t_0$, respectively.
As the parameter passes through $t_0$, the ordering between the green
curve and the blue curve is exchanged, and this is precisely the reason
for the appearance of the kink; the red curve $\lambda_{-1}$, by
contrast, maintains strict sign separation from all other branches
throughout, does not participate in any crossing, and its corresponding sorting label remains unchanged over the entire interval. More precisely, after reordering, the sorted eigenvalue functions~$\mathfrak{s}_t(1)$ and $\mathfrak{s}_t(3)$ each undergo a label
switch at $t_0$, developing a kink at the crossing point and ceasing
to be~$C^1$ as functions of~$t$; whereas~$\mathfrak{s}_t(2)$ follows
the blue curve throughout, undergoes no label switch, and retains~$C^1$ regularity on the entire interval.

Consequently, the $C^1$ regularity of the local branches alone is not sufficient to imply the $C^1$ regularity of the sorted spectral tuple
\begin{align}
t\longmapsto \mathfrak{s}_t.
\end{align}
Controlling the latter requires, in addition, the $\arsinh$-Lipschitz estimate along each branch from Corollary~\ref{cor:asinhLip_branch_prelim} and the sorting-stability conclusion of Lemma~\ref{lem:sorting-stability}. This is precisely what is achieved in the proof of Theorem~\ref{thm:MonLip_prelim}.

\section{Discussion and outlook: extension beyond the Dirac operator}\label{sec:discussion}

We conclude by pointing out that the arguments developed in Sections~\ref{sec:C0weight} and~\ref{sec:C1weight} are not tied to the spin Dirac operator itself, but depend only on a small collection of structural properties that remain valid for a much broader class of elliptic operators.

More precisely, let $(M^n,g)$ be a closed Riemannian manifold, let $E\to M$ be a Hermitian complex vector bundle, and let
\begin{align}
P:\Gamma(E)\longrightarrow \Gamma(E)
\end{align}
be a fixed first-order formally self-adjoint elliptic operator, see e.g. \cite{LawsonMichelsohn1989spin}. For a weight
\begin{align}
A\in \mathcal P_{\Lambda_1,\Lambda_2},
\end{align}
one may consider the weighted eigenvalue problem
\begin{align}
P\psi=\lambda A\psi,
\end{align}
or, equivalently, the conjugated operator
\begin{align}
\widetilde P(A):=A^{-1/2}PA^{-1/2}.
\end{align}

To make the scope of the argument transparent, we summarize below the structural ingredients that are actually used in the paper.

\begin{table}[ht]
\centering
\caption{Structural properties used in the proofs and their level of specificity.}\label{tab:general-elliptic-extension}
\begin{tabular}{|p{0.34\textwidth}|p{0.32\textwidth}|p{0.18\textwidth}|}
\hline
\textbf{Property} & \textbf{Where it is used} & \textbf{Dirac-specific?} \\
\hline
First-order ellipticity and compact resolvent & Corollary~\ref{cor:typeA_discrete_family}; compactness of the resolvent & No \\
\hline
$L^2$-self-adjointness & Theorem~\ref{thm:isometric_weighted_dirac_realparam}(2); Theorem~\ref{thm:WHF_prelim} & No \\
\hline
Elliptic graph-norm estimate
& Corollary~\ref{cor:typeA_discrete_family}& No \\
\hline
The elliptic operator $P$ is independent of $t$& Proof of Theorem~\ref{thm:WHF_prelim} & No \\
\hline
Constancy of the kernel dimension & Argument proving $k(A)=0$ in Theorem~\ref{thm: C0 main thm} & Needs verification\\
\hline
Weyl asymptotics, ensuring that the spectrum defines an element of $\mathfrak{Mon}$ & Definition~\ref{def:ordered_spectrum_counting} and the global spectral parametrization & No \\
\hline
\end{tabular}
\end{table}

As Table~\ref{tab:general-elliptic-extension} shows, almost all steps in the proof rely only on abstract elliptic theory rather than on special features of spin geometry. The only point that deserves separate comment is the stability of the kernel. In the present setting, however, this remains true, since
\begin{align}
\ker \widetilde P(A)=A^{1/2}\ker P,
\end{align}
and therefore
\begin{align}
\dim \ker \widetilde P(A)=\dim \ker P,
\qquad
\text{for all }A\in\mathcal P_{\Lambda_1,\Lambda_2}.
\end{align}
Thus the kernel dimension remains constant as long as the underlying operator $P$ is fixed and only the weight $A$ varies.

It follows that $\widetilde P(A)$ retains exactly the properties used in our proofs: it is self-adjoint on the fixed Hilbert space $L^2(M,E)$, has compact resolvent, satisfies the corresponding elliptic graph-norm equivalence, and admits the same kernel control as in the Dirac case. Consequently, the proofs of Theorems~\ref{thm: C0_CONT} and~\ref{thm:C1_CONT} extend, with no essential change, from $\D$ to an arbitrary fixed first-order formally self-adjoint elliptic operator $P$.

We refer to~\cite{Bar1999zero,BarBandara2022boundary} for further examples of general first-order elliptic operators. These typical examples covered by this extension include twisted Dirac operators, the de~Rham--Hodge operator and signature-type operators.

We emphasize, however, that the present paper is deliberately formulated for the Dirac operator. This is not because the method fails in greater generality, but because the Dirac operator comes with a rich geometric and analytic background and with concrete motivation of independent interest, as explained in the Introduction. By contrast, although the abstract extension above is mathematically straightforward, we do not at present see equally compelling background motivation for developing the full general theory within the body of this paper. For this reason, we have chosen to keep the main exposition in the geometrically most natural and best-motivated Dirac setting.

\

\textbf{Conflict of Interest.} The authors have no conflicts to disclose.

\

\textbf{Acknowledgement.} Z.Q thanks Prof. Tongzhu Li for constant support. R. W. thanks Prof. Zuoqin Wang for helpful conversations on this topic.

\

\textbf{Data Availability Statement.} Data sharing not applicable to this article as no datasets were generated or
analyzed during the current study.

\printbibliography[
heading=bibintoc,
title={References}
]
\end{document}